\pdfoutput=1
\documentclass[10pt,reqno]{amsart}
\usepackage[margin=1.1in]{geometry} %

\usepackage{amssymb}

\usepackage{tikz}
\usepackage{tikz-cd}
\usetikzlibrary{decorations.pathreplacing} %
\usepackage[style=alphabetic,maxalphanames=4,maxnames=10,backend=biber,
            useprefix=true,hyperref=true]{biblatex}
\usepackage{mathtools} %
\usepackage{enumitem}
\usepackage{verbatim} %

\definecolor{darkgreen}{rgb}{0,0.45,0}
\usepackage{hyperref}
\hypersetup{colorlinks,urlcolor=blue,citecolor=darkgreen,linkcolor=darkgreen,
linktocpage,bookmarksnumbered}

\usepackage[capitalize]{cleveref}

\newcommand{\Cb}{\mathbb{C}}
\newcommand{\Rb}{\mathbb{R}}
\newcommand{\Sb}{\mathbb{S}}
\newcommand{\Zb}{\mathbb{Z}}

\DeclareMathOperator{\Hom}{\mathrm{Hom}}
\DeclareMathOperator{\Ext}{\mathrm{Ext}}

\newcommand{\inv}{^{-1}}
\newcommand{\id}{\term{id}}

\DeclareMathOperator{\Aut}{\mathrm{Aut}}
\DeclareMathOperator{\BAut}{\mathrm{BAut}}
\DeclareMathOperator{\BAutpl}{\mathrm{BAut}_{1}}
\newcommand\BAutpln[1]{\mathrm{BAut}_1^{#1}}

\DeclareMathOperator{\ev}{\textup{\textbf{ev}}}
\DeclareMathOperator{\evid}{\ev_{\id}}
\DeclareMathOperator{\Type}{\mathcal{U}}

\newcommand{\symId}{\widetilde{\mathrm{Id}}}
\DeclareMathOperator{\const}{\mathrm{const}}
\DeclareMathOperator{\refl}{\mathrm{refl}}
\newcommand\invrefl{{\mathrm{refl}^*}} %

\DeclareMathOperator{\ap}{\term{ap}}
\DeclareMathOperator{\merid}{\term{merid}}
\DeclareMathOperator{\pr}{\term{pr}}
\newcommand{\pcomp}[2]{{#1}_{[#2]}} %

\newcommand\on\operatorname

\newcommand\ab{\on{Ab}}

\newcommand{\term}[1]{\mathrm{{#1}}}
\DeclarePairedDelimiter\Trunc{\lVert}{\rVert} %
\DeclarePairedDelimiter\trunc{\lvert}{\rvert} %
\newcommand{\pt}{\term{pt}}

\newcommand\HSpace[1]{\on{HSpace}(#1)} %
\newcommand{\ZHSpace}[2]{\on{HSpace}^{#1}(#2)}
\newcommand{\GUP}{\on{GUP}} %

\newcommand{\smin}{\term{sm}}
\newcommand{\auxl}{\term{auxl}}
\newcommand{\auxr}{\term{auxr}}
\newcommand{\gluel}{\term{gluel}}
\newcommand{\gluer}{\term{gluer}}

\newcommand\EM[2]{\on{K}(#1,#2)} %

\newcommand{\pto}{\to_*}
\newcommand{\lra}       {\longrightarrow}
\newcommand{\llra}[1]   {\stackrel{#1}{\lra}}  %
\newcommand{\jeq}{\equiv} %

\theoremstyle{plain}
\newtheorem{thm}{Theorem}[section]
\newtheorem{lem}[thm]{Lemma}
\newtheorem{cor}[thm]{Corollary}
\newtheorem{prop}[thm]{Proposition}

\theoremstyle{definition}
\newtheorem{defn}[thm]{Definition}
\newtheorem{ex}[thm]{Example}

\theoremstyle{remark}
\newtheorem{rmk}[thm]{Remark}

\Crefname{thm}{Theorem}{Theorems}
\Crefname{prop}{Proposition}{Propositions}
\Crefname{lem}{Lemma}{Lemmas}

\usepackage[textsize=tiny,colorinlistoftodos]{todonotes}

\newcommand{\define}[1]{\textbf{\boldmath{#1}}}

\newcommand{\period}{\rlap{\hspace{2pt}.}}
\newcommand{\comma}{\rlap{\hspace{2pt},}}

\newcommand{\ct}{%
  \mathchoice{\mathbin{\raisebox{0.5ex}{$\displaystyle\centerdot$}}}%
             {\mathbin{\raisebox{0.5ex}{$\centerdot$}}}%
             {\mathbin{\raisebox{0.25ex}{$\scriptstyle\,\centerdot\,$}}}%
             {\mathbin{\raisebox{0.1ex}{$\scriptscriptstyle\,\centerdot\,$}}}%
}

\newcommand\linv{\backslash}
\newcommand\rinv{/}

\newcommand\dual[1]{#1^*}
\newcommand\dualid{\dual{\id}} %

\makeatletter
\renewcommand\subsection{\@startsection{subsection}{2}%
  \z@{.5\linespacing\@plus.7\linespacing}{-.5em}%
  {\normalfont\bfseries\boldmath}}
\makeatother

\DeclareRobustCommand{\SkipTocEntry}[5]{}

\begin{document}

\title{Central H-spaces and banded types}

\author[Buchholtz]{Ulrik Buchholtz}
\address{University of Nottingham, Nottingham, United Kingdom}
\email{ulrik.buchholtz@nottingham.ac.uk}

\author[Christensen]{J. Daniel Christensen}
\address{University of Western Ontario, London, Ontario, Canada}
\email{jdc@uwo.ca}

\author[Flaten]{Jarl G.\ Taxerås Flaten}
\address{University of Western Ontario, London, Ontario, Canada}
\email{jtaxers@uwo.ca}

\author[Rijke]{Egbert Rijke}
\address{University of Ljubljana, Ljubljana, Slovenia}
\email{e.m.rijke@gmail.com}

\begin{abstract}
  We introduce and study \emph{central} types, which are generalizations of Eilenberg--Mac~Lane spaces.
  A type is central when it is equivalent to the component of the identity among its own self-equivalences.
  From centrality alone we construct an infinite delooping in terms of a tensor product of \emph{banded types}, which are the appropriate notion of torsor for a central type.
  Our constructions are carried out in homotopy type theory, and therefore hold in any \(\infty\)-topos.
  Even when interpreted into the \(\infty\)-topos of spaces, our approach to constructing these deloopings is new.

  Along the way, we further develop the theory of H-spaces in homotopy type theory, including their relation to \emph{evaluation fibrations} and Whitehead products.
  These considerations let us, for example, rule out the existence of H-space structures on the \(2n\)-sphere for \(n > 0\).
  We also give a novel description of the moduli space of H-space structures on an H-space.
  Using this description, we generalize a formula of Arkowitz--Curjel and Copeland for counting the number of path components of this moduli space.
  As an application, we deduce that the moduli space of H-space structures on the \(3\)-sphere is \(\Omega^6 \Sb^3\).
\end{abstract}

\date{April 2, 2025}

\maketitle

\tableofcontents

\section{Introduction}

In this paper we study H-spaces and their deloopings.
By working in homotopy type theory, our results apply to any $\infty$-topos~\cite{dBB,dB,KL,LS,Shu}.
While most of the theory we develop is new to homotopy type theory, much of the theory related to H-spaces in \cref{sec:H-spaces} is inspired by classical topology.
However, our constructions in \cref{sec:central-types,sec:bands-and-torsors} are new even for the $\infty$-topos of spaces.

A key concept is that of a central type.
A pointed type $A$ is \textbf{central} if the map
\[ f \longmapsto f(\pt) \ : \ \pcomp{(A \to A)}{\id} \longrightarrow A \]
which evaluates a function at the base point of \(A\) is an equivalence.
Here $\pcomp{(A \to A)}{\id}$ denotes the identity component of the type of all
self-maps of $A$, and $\pt$ denotes the base point of $A$.
Since the domain of this map is a connected H-space (see \cref{defn:H-space}(2)), so is any central type.
We give a list of equivalent conditions for a connected H-space to be central in \cref{prop:central-tfae}, and one of them is that the type \(A \pto A\) of pointed self-maps is a set.
It follows, for example, that every Eilenberg--Mac~Lane space $\EM{G}{n}$,
with $G$ abelian and $n \geq 1$, is central.
We show in \cref{ss:products-EM-spaces} that some, but not all, products of
Eilenberg--Mac~Lane spaces are central.
In \cref{ss:stable-n-groups}, we show that
every truncated, central type with at most two non-zero homotopy groups,
both of which are finitely presented,
is a product of Eilenberg--Mac~Lane spaces.
In general, we don't know whether every central type is
a product of Eilenberg--Mac~Lane spaces, and we leave this as an open question.

Our first result is:\nopagebreak
\theoremstyle{plain}
\newtheorem*{thm:uniquely-deloopable}{\cref{thm:uniquely-deloopable}}
\begin{thm:uniquely-deloopable}
Let $A$ be a central type.  Then $A$ has a unique delooping.
\end{thm:uniquely-deloopable}
The key ingredient of this result---and much of the paper---is that we have
a concrete description of the delooping of $A$.
It is given by the type $\BAutpl(A) :\jeq \sum_{X : \Type} \, \Trunc{A = X}_0$
of types \emph{banded} by $A$,
which is the 1-connected cover of $\BAut(A)$.
As an example, since $\EM{G}{n}$ is central for $G$ abelian and $n \geq 1$,
this gives an alternative way to define $\EM{G}{n+1}$ in terms of $\EM{G}{n}$,
as previously indicated by the first author~\cite{BModalHoTT2019}.
Banded types are denoted \(X_p\), where \(p : \Trunc{A = X}_0\) is the band.

We also show:
\newtheorem*{thm:Omega-equiv}{\cref{thm:Omega-equiv}}
\begin{thm:Omega-equiv}
  Let $A$ be a central type.
  Every pointed map $f : A \pto A$ has a
  unique delooping
  \[ X_p \longmapsto \pcomp{(X \to A)}{\dual{f} \circ \tilde{p}\inv} \ : \ \BAutpl(A) \longrightarrow_* \BAutpl(A). \]
\end{thm:Omega-equiv}
Here $X$ is a type banded by $p$, while \(\pcomp{(X \to A)}{\dual{f} \circ \tilde{p}\inv}\) denotes
a component of a mapping type, whose banding is explained in \cref{defn:Bf}.
It follows from the theorem that the type of pointed self-maps of $\BAutpl(A)$ is a set,
since it is equivalent to $A \pto A$.

One of the motivations for studying $\BAutpl(A)$ is that one can define
a \emph{tensoring} operation.
Given two banded types $X$ and $Y$ in $\BAutpl(A)$, the type $\dual{X} = Y$
has a natural banding, where $\dual{X}$ denotes a certain dual of $X$.
We write $X \otimes Y$ for this banded type, and show in \cref{thm:BAutpl-H-space}
that it makes $\BAutpl(A)$ into an abelian H-space.
Combined with \cref{thm:uniquely-deloopable}, \cref{thm:Omega-equiv}, and
the characterization of central types mentioned earlier, we therefore deduce:

\newtheorem*{cor:central-infinite-loop-space}{\cref{cor:central-infinite-loop-space}}
\begin{cor:central-infinite-loop-space}
For a central type $A$, the type $\BAutpl(A)$ is again central.
Therefore, $A$ is an infinite loop space, in a unique way.
Moreover, every pointed map $A \pto A$ is infinitely deloopable, in a unique way.
\end{cor:central-infinite-loop-space}

Our tensoring operation gives a new description of the H-space structure
on $\EM{G}{n}$, which will be helpful for calculations of Euler classes in work in progress
and is what originally motivated this work.

We also give an alternate description of the delooping of a central type $A$ as a certain type of $A$-torsors (\cref{ss:bands-and-torsors}), and give an analogous description of $\EM{G}{1}$ for any group $G$ (\cref{ss:G-torsors}).

\medskip

To prove the above results, we first need to further develop the theory of H-spaces.
One difference between our work and classical work in topology is that we
emphasize the moduli space $\HSpace{A}$ of H-space structures on a pointed type $A$,
rather than just the set of components.
For example, we prove:

\newtheorem*{thm:space-of-H-space-structures}{\cref{thm:space-of-H-space-structures}}
\begin{thm:space-of-H-space-structures}
Let \( A \) be an H-space such that for all $a : A$, the map $a \cdot -$ is an equivalence.
Then the type $\HSpace{A}$ of H-space structures on \( A \) is equivalent to the type
\( A \land A \to_* A \) of pointed maps.
\end{thm:space-of-H-space-structures}

This generalizes a classical formula of Arkowitz--Curjel and Copeland, which
plays a key role in classical results on the number of H-space structures on various spaces.
The classical formula only determines the path components of the type of H-space structures,
while our formula gives an equivalence of types.
From our formula it immediately follows, for example, that the type of H-space structures
on the \(3\)-sphere is \( \Omega^6 \Sb^3 \).
The proof of \cref{thm:space-of-H-space-structures} uses \emph{evaluation fibrations},
which generalize the map appearing in the definition of ``central.''
In fact, these evaluation fibrations play an important role in much of the paper, and underlie our main results about central types.
We also relate the existence of sections of an evaluation fibration to the vanishing of Whitehead products.
Related considerations let us show that no even spheres besides $\Sb^0$ admit H-space structures.

In \cref{prop:hspace-pointed-section-comp} we show that every central type
has a unique H-space structure, in the strong sense that the type $\HSpace{A}$
is contractible.
We prove several results about types with unique H-space structures.
For example, we show that such H-space structures are associative and coherently abelian,
and that every pointed self-map is an H-space map, a weaker version of the
delooping above.
We also give an example, pointed out to us by David Wärn, which shows that not every type with a unique H-space structure is central (\cref{ex:unique-H-space-not-central}).

We note that these results rely on us defining ``H-space'' to
include a coherence between the two unit laws (see \cref{defn:H-space}).

\addtocontents{toc}{\SkipTocEntry}
\subsection*{Relation to other work}

The concept of ``band'' (``lien'' in French) originates with Giraud~\cite[Ch.~IV]{Giraud1971};
see also~\cite[Sec.~7.2.2]{Lurie-HTT}.
Dwyer and Wilkerson~\cite{DwyerWilkerson1995} introduced the notion of centrality in
the context of $p$-compact groups under the name of ``abelian''.

The proof of Theorem~1.7 in \cite{Bousfield1996} shows that in classical homotopy theory
every central space is a product of Eilenberg--Mac~Lane spaces (a ``GEM''),
generalizing our \cref{thn:central-implies-GEM}.
However, this argument depends on a lot of machinery that is not available in homotopy
type theory, and we don't know if it is true in any $\infty$-topos.
In \cite[Prop.~5.9]{Bousfield1997}, Bousfield shows that a GEM $A$ is central if
and only if it has \emph{transitory} homotopy groups:
$\Hom(\pi_mA, \pi_nA) = 0$ for $n \geq m+1$ and
$\Ext(\pi_mA, \pi_nA) = 0$ for $n \geq m+2$.
Again, we don't know if this is true in our setting.

In classical homotopy theory, our result that any central type has a unique delooping
(\cref{thm:uniquely-deloopable}) can be strengthened to the statement that
any pointed, connected space with a unique H-space structure has a unique delooping.
The proof proceeds by showing that such a space has a unique $A_{\infty}$-structure.
(We thank the referee for sketching this argument for us.)
However, the theory of $A_{\infty}$-structures is not available in homotopy type theory,
due to the required coherences.
In homotopy type theory, David Wärn~\cite{Warn23} has used a different approach to
prove the stronger result.

\addtocontents{toc}{\SkipTocEntry}
\subsection*{Formalization}
Many of the main results of this paper have been formalized in the Coq proof assistant using the Coq-HoTT library~\cite{coqhott}.
This includes much of the basic theory that we develop related to H-spaces, and these results have already been merged into the Coq-HoTT library under the \(\texttt{Homotopy.HSpace}\) namespace.
Notably, \cref{thm:space-of-H-space-structures} has been formalized as \href{https://github.com/HoTT/Coq-HoTT/blob/1032ee0e7d2b4838b068944443c04318813bbcd9/theories/Homotopy/HSpace/Moduli.v\#L130}{\(\texttt{equiv\_cohhspace\_ppmap}\)}, modulo the smash-hom adjunction which has been formalized by Floris van Doorn in Lean 2~\cite{vanDoorn2018}.
In a separate repository~\cite{central}, we have also formalized \cref{ex:EM}, \cref{thm:uniquely-deloopable}, \cref{thm:BAutpl-H-space}, and \cref{cor:central-infinite-loop-space}, along with their dependencies.
In that repository, we have also formalized the equivalence stated in \cref{thm:Omega-equiv},
but with both sides restricted to pointed equivalences, which is all that is needed
for \cref{cor:central-infinite-loop-space}.
These results will also be submitted to the Coq-HoTT library.

\addtocontents{toc}{\SkipTocEntry}
\subsection*{Outline}
In \cref{sec:H-spaces}, we give results about H-spaces which do not depend on centrality,
including a description of the moduli space of H-space structures,
results about Whitehead products and H-space structures on spheres,
and results about unique H-space structures.
In \cref{sec:central-types}, we define central types,
show that central types have a unique H-space structure,
give a characterization of which H-spaces are central,
and prove other results needed in the next section.
\cref{sec:bands-and-torsors} is the heart of the paper.
It defines the type $\BAutpl(A)$ of bands for a central type $A$,
shows that it is a unique delooping of $A$,
proves that it is an H-space under a tensoring operation,
and shows that central types and their self-maps
are uniquely infinitely deloopable.
We also give the alternate description of the delooping in terms of $A$-torsors.
Finally, \cref{sec:examples} gives examples and non-examples of central types,
mostly related to Eilenberg--Mac~Lane spaces and their products.
It also contains a construction of $K(G,1)$ as a type of torsors,
and defines an H-space structure on this type when $G$ is abelian,
paralleling the H-space structure on $\BAutpl(A)$.
It ends by showing that
every truncated, central type with at most two non-zero homotopy groups,
both of which are finitely presented,
is a product of Eilenberg--Mac~Lane spaces.

\addtocontents{toc}{\SkipTocEntry}
\subsection*{Notation and conventions}
In general, we follow the notation used in~\cite{HoTTBook}.
For example, we write path composition in diagrammatic order:
given paths \(p : x = y\) and \(q : y = z\), their composite is \( p \ct q\).
The reflexivity path is written \(\refl\).

We write $\Type$ for a fixed univalent universe of types, and frequently
make use of univalence.
We also use function extensionality without always explicitly mentioning it.

Given a type \(A\) and an element \(a : A\), we write \((A,a)\) for the type $A$ pointed at $a$.
If \(A\) is already a pointed type with unspecified base point, then we write \(\pt\) for the base point.
If \(A\) and \(B\) are pointed types, and \(f, g : A \to_* B\) are pointed maps, then \(f =_* g\) is the type of pointed homotopies between \(f\) and \(g\).

If \(A\) is an H-space, then we write \(x \cdot y\) for the product of two elements \(x, y : A\) (unless another notation for the multiplication is given).
For a pointed type $A$, we write $\HSpace{A}$ for the type of H-space structures
on $A$ with the base point as the identity element (\cref{defn:H-space}).

We write \(\Sb^n\) for the \(n\)-sphere.

\addtocontents{toc}{\SkipTocEntry}
\subsection*{Acknowledgements.}
We would like to thank David Jaz Myers for many lively discussions on the content of this paper, especially related to bands and torsors.
We also thank David Wärn for fruitful discussions and for early drafts
of his related paper~\cite{Warn23}.
Finally, we thank the referee for many comments that helped to improve the paper,
in particular by relating it to other work in classical homotopy theory.

Egbert Rijke gratefully acknowledges the support by the Air Force Office of Scientific Research through grant FA9550-21-1-0024, and support by the Slovenian Research Agency research programme P1-0294.
Dan Christensen and Jarl Flaten both acknowledge the support
of the Natural Sciences and Engineering Research Council of Canada
(NSERC), RGPIN-2022-04739.

\section{H-spaces and evaluation fibrations}\label{sec:H-spaces}

In \cref{ss:H-space-structures}, we begin by recalling the notion of a (coherent) \emph{H-space structure} on a pointed type \(A\), give several equivalent
descriptions of the type of H-space structures, and prove basic results
that will be useful in the rest of the paper.

In \cref{ss:extensions-and-whitehead-products},
we discuss the type of pointed extensions of a map $B \vee C \pto A$ to $B \times C$,
and show that the type of H-space structures on $A$ is equivalent to the type of
pointed extensions of the fold map.
We relate the existence of extensions to the vanishing of Whitehead products,
and use this to show that there are no H-space structures on even spheres except \(\Sb^0\).
In addition, we show that for any \(n\)-connected H-space \(A\),
the Freudenthal map \( \pi_{2n+1}(A) \to \pi_{2n+2}(\Sigma A) \) is an isomorphism,
not just a surjection.

In \cref{ss:evaluation-fibrations}, we study \emph{evaluation fibrations}.
We show that for a left-invertible H-space $A$, various evaluation fibrations
are trivial, and use this to show that the type of H-space structures is equivalent to
$A \wedge A \pto A$, generalizing a classical formula of Arkowitz--Curjel and Copeland.
It immediately follows, for example, that the type of H-space structures on the \(3\)-sphere is \( \Omega^6 \Sb^3 \).

\cref{ss:unique-H-space-structures} is a short section which studies the
case when the type of H-space structures is contractible.
We stress that this is not the same as \(\HSpace{A}\) having a single component,
which is what is classically meant by ``\(A\) has a unique H-space structure.''
This situation is interesting in its own right.
We show that such H-space structures are associative and coherently abelian,
and we prove that all pointed self-maps of $A$ are automatically H-space maps.

\subsection{H-space structures}\label{ss:H-space-structures}
We begin by giving the notion of H-space structure that we will consider in this paper.
\pagebreak[2]

\begin{defn} \label{defn:H-space}
  Let \(A\) be a pointed type.
  \begin{enumerate}
  \item A \textbf{non-coherent H-space structure} on \(A\) consists of a binary operation \( \mu : A \to A \to A \), a \textbf{left unit law} \( \mu_l : \mu(\pt, -) = \id_A \) and a \textbf{right unit law} \( \mu_r : \mu(-, \pt) = \id_A \).
  \item A \textbf{(coherent) H-space structure} on \(A\) consists of a non-coherent H-space structure \(\mu\) on \(A\) along with a \textbf{coherence} \( \mu_{lr} : \mu_l(\pt) =_{\mu(\pt, \pt) = \pt} \mu_r(\pt) \).
  \item We write \define{$\HSpace{A}$} for the type of (coherent) H-space structures on $A$.
  \end{enumerate}
  When the H-space structure is clear from the context we may write \(x \cdot y :\jeq \mu(x,y) \).
  Any H-space structure yields a non-coherent H-space structure by forgetting the coherence.
  Suppose \(A\) has a (non)coherent H-space structure \(\mu\).
  \begin{enumerate}[resume]
  \item  If \( \mu(a,-) : A \to A \) is an equivalence for all \( a : A\), then \( \mu \) is \textbf{left-invertible}, and we write \(x \linv y :\jeq \mu(x,-)^{-1}(y)\).
    \textbf{Right-invertible} is defined dually, and we write \(x \rinv y :\jeq \mu(-,y)^{-1}(x)\).
  \item The \textbf{twist} \( \mu^T \) of \( \mu \) is the natural (non)coherent H-space structure with operation
    \[ \mu^T(a_0, a_1) :\jeq \mu(a_1, a_0). \]
  \end{enumerate}
\end{defn}

When we say ``H-space'' we mean the coherent notion---we will only say ``coherent'' for emphasis.
The notion of H-space structure considered in~\cite[Def.~8.5.4]{HoTTBook} corresponds to our \emph{non-coherent} H-space structures.
While many constructions can be carried out for non-coherent H-spaces (such as the Hopf construction), the coherent case is more natural for our purposes.

The type of H-space structures on a pointed type has several equivalent descriptions.

\begin{prop}\label{prop:H-space-structures}
Let $A$ be a pointed type.  The following types are equivalent:
\begin{enumerate}
\item The type $\HSpace{A}$ of H-space structures on $A$.
\item The type of pointed sections of the pointed map \( \ev : (A \to A, \id_A) \to_* A \)
      which sends an unpointed map $f$ to $f(\pt)$.
\item The type of families $\mu : \prod_{a:A} (A, \pt) \to_* (A,a)$ of pointed maps
      equipped with a pointed homotopy $\mu(\pt) =_* \id_A$.
\end{enumerate}
Moreover, the type of non-coherent H-space structures on $A$ is
equivalent to the type of families $\mu : \prod_{a:A} (A, \pt) \to_* (A,a)$
of pointed maps equipped with an unpointed homotopy $\mu(\pt) = \id_A$.
\end{prop}

\begin{proof}
All of these claims are simply reshuffling of data combined with function
extensionality.
For example, given a pointed section $s$ of $\ev$, the underlying map
of $s$ gives the binary operation $\mu$, the pointedness of $s$ gives
the left unit law $\mu_l$, the homotopy $\ev \circ s = \id_A$ gives
the right unit law, and the pointedness of that homotopy gives the coherence.
The data in (3) is almost identical.
In particular, it is the pointedness of the homotopy $\mu(\pt) =_* \id_A$
that corresponds to coherence, and omitting this gives the description of
non-coherent H-space structures.
\end{proof}

\begin{rmk}\label{rmk:homogeneous}
Note that $A$ is left-invertible if and only if the maps in (3) are equivalences.
We say that $A$ is a \textbf{homogeneous type} if it is equipped with a
family $\mu : \prod_{a:A} (A, \pt) \simeq_* (A,a)$, and so we see that every
left-invertible H-space is homogeneous.
\end{rmk}

We have the following converse.

\begin{lem}\label{lem:non-coherent-to-coherent}
  Let $A$ be a pointed type equipped with a family
  $\mu : \prod_{a:A} (A, \pt) \to_* (A,a)$
  such that $\mu(\pt)$ is an equivalence.
  Then $A$ can be given the structure of an H-space.
\end{lem}

\begin{proof}
  The new family defined by $\mu'(a) = \mu(a) \circ \mu(\pt)^{-1}$ has
  the property that $\mu'(\pt) =_* \id_A$, and therefore gives an H-space
  structure on $A$.
\end{proof}

Note that this lemma shows that every homogeneous type can be given the
structure of a (left-invertible) H-space, and that every non-coherent
H-space can be given the structure of an H-space.
In the latter case, since the original $\mu(\pt)$ is equal to the identity
map (as unpointed maps), the new H-space retains the same binary operation
$\mu$ and left unit law $\mu_l$, but has a different right unit law $\mu_r$.
While the types of non-coherent and coherent H-space structures on a pointed type are logically equivalent,
they are not generally equivalent as types~(see \cref{rmk:non-coherent-HSpace}).

\medskip

We'll be interested in abelian and associative H-spaces later on.

\begin{defn}\label{defn:abelian}
  Let \(A\) be an H-space with multiplication \(\mu\).
  \begin{enumerate}
  \item If there is a homotopy \( h : \prod_{a, b} \mu(a,b) = \mu(b,a) \), then \(\mu\) is \textbf{abelian}.
  \item If \(\mu = \mu^T \) in \(\HSpace{A}\), then \(\mu\) is \textbf{coherently abelian}.

  \item If there is a homotopy \( \alpha : \prod_{a,b,c : A} \mu(\mu(a,b),c) = \mu(a, \mu(b,c))\), then \(\mu\) is \textbf{associative}.
  \end{enumerate}
\end{defn}

The following lemma gives a convenient way of constructing abelian H-space structures, and will be used in \cref{thm:BAutpl-H-space}.

\begin{lem}\label{lem:symmetry}
  Let \(A\) be a pointed type with a binary operation \(\mu\), a symmetry \(\sigma_{a,b} : \mu(a, b) = \mu(b,a) \) for every \(a, b : A\) such that \(\sigma_{\pt, \pt} = \refl\), and a left unit law \(\mu_l : \mu(\pt, -) = \id_A\).
  Then \(A\) becomes an abelian H-space with the right unit law induced by symmetry.
\end{lem}

\begin{proof}
  For $b : A$, the right unit law is given by the path $\sigma_{b,\pt} \ct \mu_l(b)$ of type $\mu(b, \pt) = b$.
  For coherence we need to show that the following triangle commutes:
  \[ \begin{tikzcd}
      \mu(\pt, \pt) \ar[dr, "\mu_l" swap] \ar[rr, "\sigma_{\pt, \pt}"] && \mu(\pt, \pt) \ar[dl, "\mu_l"] \\
      & \pt \period
    \end{tikzcd} \]
  By our assumption that \(\sigma_{\pt, \pt} = \refl\), the triangle is filled \(\refl_{\mu_l}\).
\end{proof}

We collect a few basic facts about H-spaces. The following lemma generalizes a result of Evan Cavallo,
who formalized the fact that unpointed homotopies between pointed maps into a \emph{homogeneous type} $A$ can be upgraded to pointed homotopies. Being a homogeneous type is logically equivalent
to being a left-invertible H-space~\cite{CavalloHomogeneous}.
Here we do not need to assume left-invertibility, and we factor this observation through a further generalization.

\begin{lem}  \label{lem:unpointed-to-pointed-homotopy}
  Let \(A\) be a pointed type, and consider the following three conditions:
  \begin{enumerate}
  \item \(A\) is an H-space.
  \item The evaluation map \((\id_A = \id_A)\to (\pt=\pt)\) sending a homotopy $h$ to $h_{\pt}$ has a section.
  \item For every pointed type $B$ and pointed maps $f,g:B\to_\ast A$, there is a map \((f=g)\to (f=_\ast g)\) which upgrades unpointed homotopies to pointed homotopies.
  \end{enumerate}
  Then (1) implies (2) and (2) implies (3).
\end{lem}

\begin{proof}
  To show that (1) implies (2), suppose that \(A\) is an H-space.
  By~\cref{prop:H-space-structures} we have a pointed section $s$ of
  \( \ev : (A \to A, \id_A) \to_* A \).
  The evaluation map in (2) is $\Omega \ev$, and has a (pointed) section $\Omega s$.

  We next show that (2) implies (3).  Let $f, g : B \pto A$ be pointed maps
  and let $H : f = g$ be an unpointed homotopy.
  By path induction on $H$, we can assume we have a single function $f : B \to A$
  with two pointings, $f_{\pt}$ and $f_{\pt}' : f(\pt) = \pt$.
  Our goal is to define a homotopy $K : f = f$ such that $K_{\pt} = r$, where $r :\jeq f_{\pt} \cdot \overline{f_{\pt}'} : f(\pt) = f(\pt)$.
  By path induction on $f_{\pt}$, we can assume that the base point of $A$ is $f(\pt)$.
  By (2), we have \(s:(f(\pt)=f(\pt))\to (\id_A=\id_A)\) such that \(s(p,f(\pt))=p\)
  for all \(p:f(\pt)=f(\pt)\).
  For $b : B$, define $K_b$ to be $s(r, f(b))$.
  Then $K_{\pt} = r$, as required.
\end{proof}

The following result is straightforward and has been formalized, so we do not
include a proof.

\begin{prop} \label{prop:pmap-hspace}
  Suppose \(A\) is a (left-invertible) H-space.
  For any pointed type \(B\), the mapping type \(B \to_* A\) based at the constant map is naturally a (left-invertible) H-space under pointwise multiplication.
  Similarly, for any type $B$, the mapping type $B \to A$ based at the constant map
  is a (left-invertible) H-space under pointwise multiplication.
 \qed
\end{prop}

In particular, if \(A\) is left-invertible then for any \(f : B \to_* A\) there is a self-equivalence of \(B \to_* A\) which sends the constant map to \(f\)---namely, the pointwise multiplication by \(f\) on the left.

Another way to produce H-space structures is via the following result:

\begin{prop}\label{prop:retract-of-hspace}
If $A$ is an H-space and $A'$ is a pointed retract of $A$, then $A'$ is an H-space.
\end{prop}

\begin{proof}
Assume we have $s : A' \pto A$, $r : A \pto A'$ and $h : r \circ s =_* \id$.
We define a multiplication on $A'$ by sending
$(a,b) : A' \times A'$ to $r(s(a) \cdot s(b))$.
The left unit law is the composite path
\[
  \begin{tikzcd}
          r(s(\pt) \cdot s(b)) \arrow[r,equals,"\ap_r \ap_{\mu(-,s(b))} s_\pt"]
   &[4em] r(\pt \cdot s(b)) \arrow[r,equals,"\ap_r \mu_l(s(b))"]
   &[2.5em] r(s(b)) \arrow[r,equals,"h_b"]
   &[1em] b \comma
  \end{tikzcd}
\]
and the right unit law is $\ap_r \ap_{\mu(s(a),-)} s_{\pt} \ct \ap_r \mu_r(s(a)) \ct h_a$.
To show coherence, we must show that these are equal when $a \jeq b \jeq \pt$.
By cancelling the common $h_{\pt}$ and removing the $\ap_r$, we see that it suffices to prove that
\[
  \ap_{\mu(-,s(\pt))} s_{\pt} \ct \mu_l(s(\pt)) = \ap_{\mu(s(\pt),-)} s_{\pt} \ct \mu_r(s(\pt)) .
\]
To prove this, we compose both sides with $s_{\pt}$ on the right, use naturality
of $\mu_l$ and $\mu_r$, coherence, and the naturality of $\ap_{\mu}$ in its two variables.
(We can also appeal to \cref{lem:non-coherent-to-coherent} to avoid this part of the
argument.)
\end{proof}

\subsection{\texorpdfstring{$(\alpha,\beta)$}{(α,β)}-extensions and Whitehead products}
\label{ss:extensions-and-whitehead-products}

We begin by defining \emph{$(\alpha,\beta)$-extensions},
and then use them to give a different description of the type of H-space structures
on a pointed type $A$.
Then we relate them to Whitehead products, and use Brunerie's computation of
Whitehead products to rule out H-space structures on even spheres.
To do this, we prove some results about Whitehead products from~\cite{Whitehead46} which relate to H-spaces.
Finally, we also show that for an $n$-connected H-space, the Freudenthal map
\( \pi_{2n+1}(A) \to \pi_{2n+2}(\Sigma A) \) is an isomorphism, not just a surjection.
None of the results in this section are used in the rest of the paper.

\begin{defn}\label{defn:alpha-beta-extension}
  Let \(\alpha : B \pto A\) and \(\beta : C \pto A\) be pointed maps.
  An \textbf{\boldmath{\( (\alpha, \beta) \)}-extension} is a pointed map \(f : B \times C \to_* A\) equipped with a pointed homotopy filling the following diagram:
  \[ \begin{tikzcd}
      B \lor C \ar[rr, "\alpha \lor \beta"] \ar[dr] && A \\
      & B \times C \period \ar[ur, "f" swap]
  \end{tikzcd} \]
\end{defn}

\begin{rmk}\label{rmk:unpointed-extensions}
  It is equivalent to consider the type of \emph{unpointed} $(\alpha,\beta)$-extensions
  consisting of unpointed maps $f : B \times C \to A$ and unpointed fillers.
  The additional data in a pointed extension is a path $f_{\pt} : f(\pt,\pt) = \pt$ and
  a 2-path that determines $f_{\pt}$ in terms of the other data.  These form a
  contractible pair.
\end{rmk}

When \(\alpha\) and \(\beta\) are maps between spheres, Whitehead instead says that \(f\) is ``of type \((\alpha, \beta)\)'' but we prefer to stress that we work with a structure and not a property.

We now relate $(\alpha, \beta)$-extensions to the following generalization of
the map $\ev$.
Given a pointed map \(f : B \pto A\), we again write $\ev$ for the map
\( (B \to A, f) \pto A \)
which evaluates at \(\pt : B\).
This map is pointed since \(f\) is.
Recall that the case where $f \jeq \id_A$ played a key role in
\cref{prop:H-space-structures}.

\begin{defn}
  Let \(e : X \pto A\) and \(g : Y \pto A\) be pointed maps.
  A \textbf{pointed lift of \boldmath{\(g\)} through \boldmath{\(e\)}} consists of a pointed map \(s : Y \pto X\) along with a pointed homotopy \(e \circ s =_* g\).
  When \(g \jeq \id\), then we recover the notion of a pointed section of \(e\).
\end{defn}

\begin{prop} \label{prop:equiv-extensions-pointed-lifts}
  Let \(\alpha : B \pto A\) and \(\beta : C \pto A\) be pointed maps.
  The type of \((\alpha,\beta)\)-extensions is equivalent to the type of pointed lifts of \(\beta\) through \(\ev : (B \to A, \alpha) \pto A\).
  \qed
\end{prop}

The proof of the statement is a straightforward reshuffling of data along
with cancellation of a contractible pair.
Diagrammatically, it gives a correspondence between the dashed arrows below, with pointed homotopies filling the triangles:
\[ \begin{tikzcd}
    B \lor C \rar["f \lor g"] \dar & A && & (B \to A, f) \dar["\ev"] \\
    B \times C \ar[ur, dashed] & && C \rar["g"] \ar[ur, dashed] & A
\end{tikzcd} \]

\begin{prop} \label{prop:extension-H-space}
  H-space structures on a pointed type \(A\) correspond to \((\id_{A}, \id_{A})\)-extensions.
\end{prop}

\begin{proof}
This follows from \cref{prop:H-space-structures,prop:equiv-extensions-pointed-lifts}.
\end{proof}

\begin{lem}\label{lem:extension-exists-H-space}
  If $A$ is an H-space, then there is an $(\alpha,\beta)$-extension for
  every pair \(\alpha : B \pto A\) and \(\beta : C \pto A\) of pointed maps.
\end{lem}

\begin{proof}
  Using naturality of the left and right unit laws and coherence, one can show that
  the map \((b,c) \mapsto \alpha(b) \cdot \beta(c) : B \times C \to A\) is an
  \((\alpha, \beta)\)-extension.
  Alternatively, observe that the $(\alpha,\beta)$-extension problem factors
  through the $(\id_A,\id_A)$-extension problem via the map
  $\alpha \times \beta : B \times C \to A \times A$.
\end{proof}

These results explain the relation between H-space structures and \((\alpha, \beta)\)-extensions, which are in turn related to Whitehead products via the next two results.
(See~\cite[Section~3.3]{Brunerie2016} for the definition of Whitehead products.)

\begin{prop}[{\cite[Corollary~3.5]{Whitehead46}}] \label{prop:alpha-beta-extension-iff-constant}
  Let \(m, n > 0\) be natural numbers and consider two pointed maps
  \(\alpha : \Sb^m \pto A\) and \(\beta : \Sb^n \pto A\).
  The type of \((\alpha, \beta)\)-extensions is equivalent to the type of witnesses that the map
  \( [\alpha, \beta] : \Sb^{m+n-1} \to_* A\) is constant (as a pointed map).
\end{prop}

\begin{proof}
  Consider the diagram of pointed maps below, where the composite of the top two maps is \([\alpha, \beta]\) and the left diamond is a pushout of pointed types by~\cite[Proposition~3.2.2]{Brunerie2016}:
  \[ \begin{tikzcd}
    & \Sb^m \lor \Sb^n \ar[dr] \ar[drr, bend left, "\alpha \lor \beta"] \\
    \Sb^{m+n-1} \ar[dr] \ar[ur] && \Sb^m \times \Sb^n \rar["f", dashed] & A \period \\
    & 1 \ar[ur] \ar[urr, bend right]
  \end{tikzcd} \]
An \((\alpha, \beta)\)-extension is the same as a pointed map $f$ along with a
pointed homotopy filling the top-right triangle.
Since the bottom-right triangle is filled by a unique pointed homotopy,
an \((\alpha, \beta)\)-extension thus corresponds exactly to the data of a filler in the
outer diagram, i.e., a homotopy witnessing that \([\alpha, \beta]\) is constant as a pointed map.
\end{proof}

With the notation of the previous proposition, we have the following:

\begin{cor}[{\cite[Corollary~3.6]{Whitehead46}}] \label{cor:H-space-whitehead-vanishes}
  Suppose \(A\) is an H-space.
  Then \([\alpha, \beta]\) is constant.
\end{cor}

\begin{proof}
  This follows from \cref{lem:extension-exists-H-space,prop:alpha-beta-extension-iff-constant}.
\end{proof}

Using the above results, we can rule out H-space structures on even spheres in positive dimensions.

\begin{prop} \label{prop:hspace-S2}
  The \(n\)-sphere merely admits an H-space structure if and only if \([\iota_n, \iota_n] = 0\).
  In particular, there are no H-space structures on the $n$-sphere when $n>0$ is even.
\end{prop}

\begin{proof}
  The implication \((\rightarrow)\) is immediate by \cref{cor:H-space-whitehead-vanishes}.
  Conversely, \cref{prop:alpha-beta-extension-iff-constant} implies that \([\iota_n, \iota_n] = 0 \) if and only if an \( (\id_{\Sb^n}, \id_{\Sb^n})\)-extension merely exists, which by \cref{prop:extension-H-space} happens if and only if \(\Sb^n\) merely admits an H-space structure.

  Finally, using that \([\iota_n, \iota_n] = 2\) in \(\pi_{2n-1}(\Sb^n)\) for even \(n > 0\)~\cite[Proposition~5.4.4]{Brunerie2016}, the above implies that \(\Sb^n\) cannot admit an H-space structure.
\end{proof}

We also record the following result and a corollary.

\begin{prop} \label{prop:H-space-freudenthal-unit-section}
  Let \(A\) be a left-invertible H-space.
  The unit \(\eta : A \to_* \Omega \Sigma A\) has a pointed retrac\-tion, given by the connecting map \( \delta : \Omega \Sigma A \to_* A \) associated to the Hopf fibration of \(A\).
\end{prop}

\begin{proof}
  Let \(\delta : \Omega \Sigma A \to_* A \) be the connecting map associated to the Hopf fibration of \(A\).
  Recall that for a loop \( p : N = N \), we have \(\delta(p) :\jeq p_*(\pt)\) where \( p_* : A \to A \) denotes transport and $A$ is the fibre above $N$.
  By definition of the Hopf fibration, a path \( \merid(a) : N =_{\Sigma A} S \) sends an element \(x\) of the fibre \(A\) to \( a \cdot x \).
  Now define a homotopy \(\delta \circ \eta = \id\) by
  \[ \delta(\eta(a)) \jeq \delta(\merid(a) \ct \merid(\pt)\inv)
  = \merid(\pt)_*\inv (\merid(a)_*(\pt))
  \jeq \pt \linv (a \cdot \pt)
  = a . \]
Finally, we promote this to a pointed homotopy using \cref{lem:unpointed-to-pointed-homotopy}.
\end{proof}

It follows that we can strengthen the conclusion of the Freudenthal suspension theorem in this situation.

\begin{cor}\label{cor:H-space-freudenthal}
  Let $n \geq 0$ and let $A$ be an $n$-connected H-space.
  Then the Freudenthal map $\eta : A \to \Omega \Sigma A$ is a $(2n+1)$-equivalence.
\end{cor}

\begin{proof}
  By Freudenthal, $\eta$ is $2n$-connected, so it induces
  an equivalence on $\pi_k$ for $k \leq 2n$ %
  and a surjection on $\pi_{2n+1}$. %
  Since $A$ is $0$-connected, it is left-invertible, so \cref{prop:H-space-freudenthal-unit-section} implies that
  the induced map on $\pi_{2n+1}$ has a retraction.
  Therefore it is a monomorphism, and hence an equivalence, as required.
\end{proof}

Since $\Sb^1$ and $\Sb^3$ are H-spaces (by~\cite[Lemma 8.5.8]{HoTTBook} and~\cite{BuchholtzRijke2018}, respectively),
we have:

\begin{cor} \label{cor:stems-stabilizes-early}
  The natural maps \( \pi_1(\Sb^1) \to \pi_2(\Sb^2) \) and \( \pi_5(\Sb^3) \to \pi_6(\Sb^4) \) are isomorphisms. \qed
\end{cor}

The fact that the unit \(\eta : A \to_* \Omega \Sigma A\) has a retraction when $A$ is a
left-invertible H-space also follows from James' reduced product construction, as shown
in~\cite{James55}.
Using~\cite{Brunerie2016}, one can see that this goes through in homotopy type theory.
However, the above argument is much more elementary.
We don't know if this argument had been observed before.

We also point out that when $n = 0$, \cref{cor:H-space-freudenthal} specializes to \cite[Theorem~4.3]{LF},
giving a new, simpler proof of that result.

\subsection{Evaluation fibrations}\label{ss:evaluation-fibrations}
We now begin our study of \emph{evaluation fibrations} and their relation to H-space structures.

\begin{defn}
  Let \( A \) be a type and \( a : \Trunc{A}_0 \).
  The \textbf{path component of \boldmath{\(a\)} in \boldmath{\(A\)}} is
  \[ \pcomp{A}{a} :\jeq \sum_{a' : A} \; (\trunc{a'}_0 = a). \]
  If \( a : A \) then we abuse notation and write \( \pcomp{A}{a} \) for \( \pcomp{A}{\trunc{a}_0} \), and in this case \( \pcomp{A}{a} \) is pointed at \( (a, \refl) \).
\end{defn}

\begin{defn} \label{defn:evaluation-fibrations}
  For any pointed map \( \alpha : B \pto A \), the \textbf{evaluation fibration (at \boldmath{\(\alpha\)})} is the pointed map \( \ev_\alpha : \pcomp{(B \to A)}{\alpha} \to_* A \) induced by evaluating at the base point of \(B\).
\end{defn}

Note that the component $\pcomp{(B \to A)}{\alpha}$ consists of maps that are merely equal to $\alpha$ as unpointed maps.
Also observe that the component \(\pcomp{(A \to A)}{\id}\) is equivalent to \(\pcomp{(A \simeq A)}{\id}\), since being an equivalence is a property of a map.
We permit ourselves to pass freely between the two.

Since pointed maps out of connected types land in the component of the base point of the codomain, we have the following consequence of \cref{prop:H-space-structures}.

\begin{cor} \label{cor:connected-hspace-pointed-sections}
  Let \(A\) be a pointed, connected type.
  The type of H-space structures on \(A\) is equivalent to the type of pointed sections of \(\evid : \pcomp{(A \simeq A)}{\id} \to_* A\). \qed
\end{cor}

For left-invertible H-spaces, various evaluation fibrations become trivial:

\begin{prop} \label{prop:ev-splits}
  Suppose \( A \) is a left-invertible H-space.
  We have a pointed equivalence over \(A\)
  \[ \begin{tikzcd}
      (A \to A) \ar[dr, "\ev" swap] \ar[rr, "\sim"] && (A \to_* A) \times A \ar[dl, "\pr_2"]   \\
      & A \comma
    \end{tikzcd} \]
  where the mapping spaces are both pointed at their identity maps.
  This pointed equivalence restricts to pointed equivalences \( (A \simeq A) \simeq_* (A \simeq_* A) \times A \) over \(A\), and \( \pcomp{(A \to A)}{\id} \simeq_* \pcomp{(A \to_* A)}{\id} \times \pcomp{A}{\pt} \) over \(\pcomp{A}{\pt}\).
\end{prop}

\begin{proof}
  Define \( e : (A \to A) \to (A \to_* A) \times A \) by
  \( e(f) :\jeq (a \mapsto f(\pt) \linv f(a), f(\pt)) \) where the first component is a pointed map in the obvious way.
  Clearly \(e\) is a map over \(A\), and moreover \(e\) is pointed.
  It is straightforward to check that the triangle above is filled by a pointed homotopy.
  (One could also apply \cref{lem:unpointed-to-pointed-homotopy}, but a direct inspection suffices in this case.)

  Finally, it's straightforward to check that \(e\) has an inverse given by
  \[ (g, a) \mapsto (x \mapsto a \cdot g(x)) . \]
  Hence \(e\) is an equivalence, as desired.
  The restrictions to equivalences and path components follow by functoriality.
\end{proof}

The hypotheses of the proposition are satisfied, for example, by connected H-spaces.

\begin{ex}
  We obtain three pointed equivalences for any abelian group \(A\) and \(n \geq 1\):
  \begin{align*}
  \big( \EM{A}{n} \to \EM{A}{n} \big)   &\simeq_* \ab(A,A) \times \EM{A}{n} , \\
  \big( \EM{A}{n} \simeq \EM{A}{n}\big) &\simeq_* \Aut_{\ab}(A) \times \EM{A}{n} , \text{ and} \\
  \pcomp{\big( \EM{A}{n} \to \EM{A}{n} \big)}{\id} &\simeq_* \EM{A}{n} .
  \end{align*}
\end{ex}

\begin{ex}
  Taking \( A :\jeq \Sb^3 \) in the previous proposition, by virtue of the H-space structure on the \( 3 \)-sphere constructed in~\cite{BuchholtzRijke2018}, we get three pointed equivalences:
  \[
    (\Sb^3 \to \Sb^3) \simeq_* \Omega^3 \Sb^3 \times \Sb^3 , \quad
    (\Sb^3 \simeq \Sb^3) \simeq_* \Omega^3_{\pm 1} \Sb^3 \times \Sb^3 , \quad \text{and} \quad
    \pcomp{(\Sb^3 \simeq \Sb^3)}{\id} \simeq_\ast \pcomp{(\Sb^3 \simeq_* \Sb^3)}{\id} \times \Sb^3 ,
  \]
  where \( \Omega^3_{\pm 1} \Sb^3 :\jeq \pcomp{(\Omega^3 \Sb^3)}{1} \sqcup \pcomp{(\Omega^3 \Sb^3)}{-1} \)
  and $1$ and $-1$ refer to the corresponding elements of $\pi_3(\Sb^3) = \Zb$.
  The self-equivalence types \(\Sb^n \simeq \Sb^n\) and \(\Sb^n \simeq_* \Sb^n\)
  are also studied
  in~\cite{CBKB}, where it is shown that these have two components, each with
  fundamental group \(\Zb/2\Zb\), for \(n\ge2\).
\end{ex}

By combining our results thus far, we obtain the following equivalence which generalizes a classical formula of~\cite[Theorem~5.5A]{Copeland59}, independently shown by~\cite{AC63}, for counting \emph{homotopy classes} of H-space structures on certain spaces.

\begin{thm}\label{thm:space-of-H-space-structures}
Let \( A \) be a left-invertible H-space.
The type $\HSpace{A}$ of H-space structures on \( A \) is equivalent to \( A \land A \to_* A \).
\end{thm}

\begin{proof}
  By \cref{prop:H-space-structures}, the type of H-space structures on \( A \) is equivalent to the type of pointed sections of \( \ev : (A \to A) \to A \).
  By \cref{prop:ev-splits}, this type is equivalent to the type of pointed sections of \( \on{pr}_2 : (A \to_* A) \times A \to A \), which are simply pointed maps \( A \to_* (A \to_* A, \id) \), where the codomain is pointed at the identity.
  The latter type is equivalent to \( A \to_* (A \to_* A) \), where the codomain is pointed at the constant map, by \cref{prop:pmap-hspace}.
  Finally, this type is equivalent to \( A \land A \to_* A \), by the smash--hom adjunction for pointed types~\cite[Theorem~4.3.28]{vanDoorn2018}.
\end{proof}

\begin{ex}
  It follows from the theorem that \( \HSpace{\Sb^1} \) is contractible and \( \HSpace{\Sb^3} \simeq \Omega^6 \Sb^3 \).
\end{ex}

\subsection{Unique H-space structures}\label{ss:unique-H-space-structures}
We collect results about H-space structures which are unique, in the sense that the type of H-space structures is contractible.
In particular, we give elementary proofs that such H-space structures are coherently abelian and associative.
Moreover, pointed self-maps of such types are always H-space maps.

There are many examples of types with a unique H-space structure.
We will see in \cref{prop:hspace-pointed-section-comp} that the ``central types''
we study in the next section have unique H-space structures.
The following proposition is another source of examples, which will
be used in \cref{ex:unique-H-space-not-central} and \cref{ss:stable-n-groups}.

\begin{prop}\label{prop:stable-types}
  Let $k \geq 0$ and let $A$ be an $(k-1)$-connected, pointed type.
  If $A$ is $(2k-1)$-truncated, then $A$ has at most one H-space structure.
  If $A$ is $(2k-2)$-truncated, then $A$ has a unique H-space structure.
\end{prop}

This also appears in~\cite{Warn23}, with a different argument.

\begin{proof}
  Assume $A$ is $(k-1)$-connected and $(2k-1)$-truncated.
  By~\cite[Corollary~2.32]{CS20}, $A \land A$ is $(2k-1)$-connected,
  and therefore $A \land A \pto A$ is contractible.
  Therefore, if $A$ has an H-space structure, then $\HSpace{A}$ is contractible,
  by \cref{thm:space-of-H-space-structures}.

  If we make the stronger assumption that $A$ is $(2k-2)$-truncated,
  then by~\cite[Theorem~6.7]{BvDR}, $A$ is an infinite loop space, so it
  is in particular an H-space.
\end{proof}

Now we show that unique H-space structures are particularly well-behaved.

\begin{lem}\label{lem:unique-abelian}
  Let \(A\) be a pointed type and suppose \(\HSpace{A}\) is contractible.
  Then the unique H-space structure \(\mu\) on \(A\) is coherently abelian.
\end{lem}

\begin{proof}
  Since \(\HSpace{A}\) is contractible, there is an identification \(\mu = \mu^T\) of H-space structures.
  (Here, $\mu^T$ is the twist, defined in \cref{defn:H-space}.)
\end{proof}

\begin{rmk}\label{rmk:symmetric-h-spaces}
  In fact, a unique H-space structure is coherently abelian in a stronger sense,
  which we now explain.
  The type of H-space structures on $A$ can be given a $\Zb/2$-action, following the general procedure of equipping a type $X$ with a $G$-action by constructing a type family $Y:BG\to\Type$ along with an equivalence $Y(\pt) \simeq X$. To construct the $\Zb/2$-action on $\HSpace{A}$ we first need to construct a $\Zb/2$-action on the identity type.
Recall that $B\Zb/2$ can be described as $\sum_{X : \Type} \, \Trunc{X = \mathbf{2}}$,
where $\mathbf{2}$ is the two-element type.
We define the \define{symmetric identity type}
  \begin{equation*}
    \symId_A:\prod_{X:B\Zb/2}A^X\to \Type
  \end{equation*}
  by $\symId_A(X,f):=\sum_{a:A} \prod_{x:X}f(x)=a$. One can easily check that $\symId_A(\mathbf{2},f)\simeq (f(0)=f(1))$, so that the symmetric identity type indeed defines a $\Zb/2$-action on the ordinary identity type.

  Now we define for any $2$-element type $X:B\Zb/2$ the type
  \begin{equation*}
    \ZHSpace{X}{A}:=\sum_{\mu:A^X\to A} \; \sum_{H:\on{unital}(\mu)} \symId(X,x\mapsto H(\const_{\pt},x,\refl)) ,
  \end{equation*}
  where $\on{unital}(\mu) :\jeq \prod_{f:X\to A}\prod_{x:X}\prod_{p:f(x)=\pt}\,\mu(f)=f(\sigma(x))$
  asserts that $\mu$ satisfies unit laws in both variables
  and $\sigma : X \to X$ is the transposition.
  It is straightforward to check that $\ZHSpace{\mathbf{2}}{A}\simeq\HSpace{A}$. The type of \define{symmetric H-space structures} on $A$ is defined to be
  \begin{equation*}
    \prod_{X:B\Zb/2}\ZHSpace{X}{A},
  \end{equation*}
  i.e., the type of fixed points of the $\Zb/2$-action on $\HSpace{A}$. If $\HSpace{A}$ is contractible, then each $\ZHSpace{X}{A}$ is contractible, and so the type of symmetric H-space structures on $A$ is also contractible.
  Since $\sigma : \mathbf{2} \to \mathbf{2}$ transports an H-space structure $\mu$ to $\mu^T$,
  it follows that a symmetric H-space structure is coherently abelian, as in \cref{lem:unique-abelian}.
  Furthermore, by applying the first projection, we obtain an operation
  \begin{equation*}
    \prod_{X:B\Zb/2} A^X\to A .
  \end{equation*}
  Put another way, we obtain an operation
  \[
    \big( \sum_{X:B\Zb/2} A^X \big) \to A
  \]
  defined on the type of \define{unordered pairs}.
  In other words, unique H-spaces are abelian in a fully coherent way.
\end{rmk}

For the next result, we need to make precise which presentation of smash products we are using.
We use the definition
from~\cite[Definition~4.3.6]{vanDoorn2018}
(see also~\cite[Definition~2.29]{CS20})
which avoids higher paths.
For pointed types $(X,x_0)$ and $(Y,y_0)$, the \define{smash product} $X \land Y$
is the higher inductive type with point constructors
$\smin : X \times Y \to X \land Y$ and
$\auxl,\, \auxr : X \land Y$,
and path constructors
$\gluel : \prod_{y : Y} \smin(x_0, y) = \auxl$
and $\gluer : \prod_{x : X} \smin(x, y_0) = \auxr$.
It is pointed by $\auxl$.
The smash product was shown to be associative in~\cite[Definition~4.3.33]{vanDoorn2018}.

\begin{prop} \label{prop:unique-pmap}
  Suppose \(A\) is a pointed type with a unique H-space structure which is left-invertible.
  Then any pointed map \(f : A \to_* A\) is an H-space map, i.e., we have
  \( f(a \cdot b) = f(a) \cdot f(b) \) for all $a, b : A$.
\end{prop}

\begin{proof}
  Let \(f : A \to_* A\) be a pointed map.
  We will define an associated map \( \nu : A \land A \to_* A \), which records
  how $f$ deviates from being an H-space map.
  We define \( \nu(\smin(a,b)) :\jeq \big(f(a) \cdot f(b)\big) \linv f(a \cdot b) \),
  $\nu(\auxl) :\jeq \pt$, and $\nu(\auxr) :\jeq \pt$.
  For $b : A$, we have a path
  $\nu(\smin(\pt, b)) \jeq \big(f(\pt) \cdot f(b)\big) \linv f(\pt \cdot b)
  = \big(\pt \cdot f(b)\big) \linv f(b)
  = f(b) \linv f(b)
  = \pt$,
  and similarly for the other path constructor.
  Since \(A\) admits a unique H-space structure, the type \(A \land A \to_* A\) is contractible by \cref{thm:space-of-H-space-structures}.
  Consequently, \(\nu\) is constant, whence for all \( a,b : A \) we have
  \( \big(f(a) \cdot f(b)\big) \linv f(a \cdot b) = \pt \), and therefore
  \[ f(a \cdot b) = f(a) \cdot f(b). \qedhere \]
\end{proof}

\begin{rmk}
Note that when $A$ and $B$ are two pointed types, each with unique H-space structures,
it is not necessarily the case that every pointed map $f : A \pto B$ is an H-space map.
For example, the squaring operation gives a natural transformation $H^2(X; \Zb) \to H^4(X; \Zb)$
which is represented by a map $\EM{\Zb}{2} \pto \EM{\Zb}{4}$.
But since squaring isn't a homomorphism, this map isn't an H-space map.
\end{rmk}

\begin{prop} \label{prop:unique-associative}
  Suppose \(A\) is a pointed type with a unique H-space structure which is left-invertible.
  Then the H-space structure is necessarily associative.
\end{prop}

\begin{proof}
  Let $a : A$.  Define a map $\nu : A \wedge A \pto A$ as follows.
  We let \( \nu(\smin(b,c)) :\jeq ((a \cdot b) \cdot c) \linv (a \cdot (b \cdot c)) \),
  $\nu(\auxl) :\jeq \pt$, and $\nu(\auxr) :\jeq \pt$.
  For $c : A$, we have a path $\nu(\smin(\pt, c)) \jeq ((a \cdot \pt) \cdot c) \linv (a \cdot (\pt \cdot c)) = (a \cdot c) \linv (a \cdot c) = \pt$,
  and similarly for the other path constructor.
  Since \(A\) admits a unique H-space structure, the type \(A \land A \to_* A\) is contractible by \cref{thm:space-of-H-space-structures}.
  Consequently, for each $a$, \(\nu\) is constant.
  It follows that for all \( a,b,c : A \) we have
  \( ((a \cdot b) \cdot c) \linv (a \cdot (b \cdot c)) = \pt \), and therefore
  \[ a \cdot (b \cdot c) = (a \cdot b) \cdot c. \qedhere \]
\end{proof}

Note that if $A \wedge A \pto A$ is contractible, then it follows from the
smash-hom adjunction that $A^{\wedge n} \pto A$ is contractible for each $n \geq 2$,
where $A^{\wedge n}$ denotes the smash power.

\section{Central types}\label{sec:central-types}

In this and the next section we focus on pointed types which we call \emph{central}.
Centrality is an elementary property with remarkable consequences.
For example, in the next section we will see that every central type is an infinite loop space (\cref{cor:central-infinite-loop-space}).
To show this, we require a certain amount of theory about central types.
We first show that every central type has a unique H-space structure.
When \(A\) is already known to be an H-space, we give several conditions which are equivalent to \(A\) being central.
From this, it follows that every Eilenberg--Mac~Lane space $\EM{G}{n}$, with $G$ abelian and $n \geq 1$, is central.
We also prove several other results which we will need in the next section.

\begin{defn}
  Let \(A\) be a pointed type.
  The \textbf{center of \boldmath{\(A\)}} is the type \( ZA :\jeq \pcomp{(A \to A)}{\id}\), which comes with a natural map \(\evid : ZA \to_* A \) (see \cref{defn:evaluation-fibrations}).
  If the map \(\evid\) is an equivalence, then \(A\) is \textbf{central}.
\end{defn}

\begin{rmk}
  The terminology ``central'' comes from higher group theory.
  Suppose \(A :\jeq B G \) is the delooping of an \(\infty\)-group \(G\).
  The center of \(G\) is the \(\infty\)-group \( ZG :\jeq \prod_{x : BG} \, (x = x) \), with delooping \(BZG :\jeq \pcomp{(BG \simeq BG)}{\id}\), which is our \(ZA\).
\end{rmk}

Central types and H-spaces are connected through evaluation fibrations:

\begin{prop} \label{prop:hspace-pointed-section-comp}
  Suppose that \(A\) is central.
  Then \(A\) admits a unique H-space structure.
  In addition, $A$ is connected, so this H-space structure is both left-
  and right-invertible.
\end{prop}

\begin{proof}
  Since \(\evid\) is an equivalence, it has a unique section.
  By \cref{cor:connected-hspace-pointed-sections}, we deduce that \(A\) has a unique H-space structure $\mu$.
  The equivalence $\evid : \pcomp{(A \to A)}{\id} \simeq A$
  implies that $A$ is connected.
  Then, since $\mu(\pt, -)$ and $\mu(-, \pt)$ are both equal to the identity,
  it follows that $\mu$ is left- and right-invertible.
\end{proof}

It follows from \cref{prop:unique-associative,lem:unique-abelian} that the unique H-space structure on a central type is associative and coherently abelian.

\begin{rmk}\label{rmk:non-coherent-HSpace}
In contrast, the type of non-coherent H-space structures on a central type $A$
is rarely contractible.
We'll show here that it is equivalent to the loop space $\Omega A$.
First consider the type of binary operations $\mu : A \to (A \to A)$
which \emph{merely} satisfy the left unit law.
This is equivalent to the type of maps $A \to \pcomp{(A \to A)}{\id}$, since $A$ is connected.
Such a map $\mu$ satisfies the right unit law if and only if the composite
$\evid \circ \mu : A \to A$ is the identity map.
In other words, $\mu$ must be a section of the equivalence $\evid$,
so there is a contractible type of such $\mu$.

The left unit law says that $\mu$ sends $\pt$ to $\id$.
After post-composing with $\evid$, it therefore says that it sends $\pt$ to $\id(\pt)$,
which equals $\pt$.
So the type of left unit laws is $\pt = \pt$, i.e., the loop space $\Omega A$.
Note that we imposed the left unit law both merely and
purely, but that doesn't change the type.
So it follows that the type of all non-coherent H-space structures on
a central type $A$ is $\Omega A$.
\end{rmk}

We give conditions for an H-space to be central, in which case the H-space structure is the unique one coming from centrality.
For the next two results, write
\[ F :\jeq \sum_{f : A \to_* A} \; \Trunc{f = \id} \]
for the fibre of \( \evid : \pcomp{(A \to A)}{\id} \to_* A \) over \( \pt : A \).
Note that the equality $f = \id$ is in the type of \emph{unpointed} maps $A \to A$.

\begin{lem}\label{lem:F}
  Suppose that \( A \) is an H-space.
  Then \(F \simeq \pcomp{(A \to_* A)}{\id} \).
\end{lem}

\begin{proof}
  This follows immediately from \cref{lem:unpointed-to-pointed-homotopy}.
\end{proof}

\begin{prop}\label{prop:central-tfae}
  Let $A$ be a pointed type.  Then the following are logically equivalent:
  \begin{enumerate}
  \item $A$ is central;
  \item $A$ is a connected H-space and $A \pto A$ is a set;
  \item $A$ is a connected H-space and $A \simeq_* A$ is a set;
  \item $A$ is a connected H-space and $A \pto \Omega A$ is contractible;
  \item $A$ is a connected H-space and $\Sigma A \pto A$ is contractible.
  \end{enumerate}
\end{prop}

\begin{proof}
  (1) $\implies$ (2):
  Assume that $A$ is central.
  Then \cref{prop:hspace-pointed-section-comp} implies that $A$ is a connected H-space.
  Since \(A\) is a left-invertible H-space, so is \(A \to_* A\), by \cref{prop:pmap-hspace}.
  Therefore all components of \(A \to_* A\) are equivalent to \(\pcomp{(A \to_* A)}{\id}\), and thus to \(F\) by \cref{lem:F}.
  Now, \(F\) is contractible since \(\evid\) is an equivalence, and consequently \(A \to_* A\) is a set since all of its components are contractible.

  (2) $\implies$ (3):
  This follows from the fact that \(A \simeq_* A\) embeds into \(A \to_* A\).

  (3) $\implies$ (1):
  If \((A \simeq_* A)\) is a set, then its component $\pcomp{(A \to_* A)}{\id}$
  is contractible.  Therefore \(F\) is contractible, by \cref{lem:F}.
  It follows that \(\evid\) is an equivalence, since \(A\) is connected.
  Hence \(A\) is central.

  (2) $\iff$ (4):  Since $A$ is a left-invertible H-space, so is $A \pto A$.
  The latter is therefore a set if and only if the component of the constant
  map is contractible, which is true if and only if the loop space $\Omega (A \pto A)$
  is contractible.  Finally, the equivalence $\Omega (A \pto A) \simeq (A \pto \Omega A)$
  shows that this is true if and only if $A \pto \Omega A$ is contractible.

  (4) $\iff$ (5):  This follows from the equivalence
  $(A \pto \Omega A) \simeq (\Sigma A \pto A)$.
\end{proof}

\begin{cor}\label{cor:retract-of-central}
If $A$ is central and $A'$ is a pointed retract of $A$, then $A'$ is central.
\end{cor}

\begin{proof}
By \cref{prop:retract-of-hspace}, $A'$ is an H-space.
Also note that $A'$ is connected, as a retract of the connected type $A$.

By \cref{prop:central-tfae}, it suffices to show that $A' \pto \Omega A'$ is contractible.
We do this by showing that it is a retract of the type $A \pto \Omega A$,
which is contractible by \cref{prop:central-tfae}.
We define a map $(A' \pto \Omega A') \to (A \pto \Omega A)$ by sending
$f$ to $(\Omega s)fr$ and a map in the other direction by sending
$g$ to $(\Omega r)gs$.
The composite sends $f$ to $(\Omega r)(\Omega s)f r s$, which is equal to $f$
because $r \circ s =_* \id$.
\end{proof}

\begin{ex}\label{ex:EM}
Consider the Eilenberg--Mac~Lane space $\EM{G}{n}$ for $n \geq 1$ and $G$ an abelian group.
It is a pointed, connected type.  Since $\EM{G}{n} \simeq \Omega \EM{G}{n+1}$, it is an H-space.
By~\cite[Theorem~5.1]{BvDR}, $\EM{G}{n} \simeq_* \EM{G}{n}$ is equivalent to the set of automorphisms of $G$.
It therefore follows from \cref{prop:central-tfae} that $\EM{G}{n}$ is central.
We will see in \cref{prop:bautpl-EM} a more self-contained proof of this result.
\end{ex}

\begin{ex}
  The fact that $\pi_4(\Sb^3) \simeq \Zb/2$~\cite{Brunerie2016} means that
  $\Sb^4 \pto \Sb^3$ is not contractible, and so $\Sb^3$ is
  not central, by \cref{prop:central-tfae}(5).
  Since this is in the stable range, it follows that $\Sb^n$ is not
  central for $n \geq 3$.
  By \cref{cor:retract-of-central}, any product $\Sb^n \times X$ with $X$ pointed
  is not central for $n \geq 3$.
\end{ex}

\begin{rmk}
  For a pointed type $A$, we have seen that $A$ being central is logically equivalent
  to $A$ being a connected H-space such that $A \simeq_* A$ is a set.
  It is natural to ask whether the reverse implication holds without the assumption
  that $A$ is an H-space.
  However, this is not the case.  Consider, for example, the pointed, connected type \(\EM{G}{1}\)
  for a non-abelian group \(G\).
  Then $\EM{G}{1} \simeq_* \EM{G}{1}$ is equivalent to the set of group automorphisms of $G$.
  If \(\EM{G}{1}\) were central, then \(G\) would be the fundamental group of an H-space,
  and would therefore be abelian.
\end{rmk}

\begin{rmk}
  In \cref{rmk:symmetric-h-spaces} we saw that every unique H-space is a symmetric H-space. In particular, every central H-space is a symmetric H-space, and therefore the binary operation of a central H-space extends to an operation on unordered pairs
  \begin{equation*}
    \begin{tikzcd}
      A^2 \arrow[dr,"\mu"] \arrow[d] \\
      \sum_{X:B\Zb/2}A^X \arrow[r,swap,"\tilde{\mu}"] & A \period
    \end{tikzcd}
  \end{equation*}
  We claim that the binary operation furthermore extends to the type of \emph{genuine} unordered pairs~\cite{B}. The type $\GUP(A)$ of \textbf{genuine unordered pairs} of elements of $A$ is defined to be the pushout
  \begin{equation*}
    \begin{tikzcd}
      B\Zb/2\times A \arrow[d,swap,"\delta"] \arrow[r,"\pi_2"] & A \arrow[d] \\
      \sum_{X:B\Zb/2}A^X \arrow[r] & \GUP(A) \comma
    \end{tikzcd}
  \end{equation*}
  where $\delta(X,a):=(X,\const_a)$. To see that $\tilde{\mu}$ extends to the genuine unordered pairs, we have to show that
  \begin{equation*}
    \tilde{\mu}_X(\const_a)=\mu(a,a)
  \end{equation*}
  for every $X:B\Zb/2$ and every $a:A$. To see this, note that both $a\mapsto \tilde{\mu}_X(\const_a)$ and $a\mapsto\mu(a,a)$ are pointed maps $A\to_\ast A$. Since $A$ is assumed to be a central H-space, the type $A\to_\ast A$ is a set, so the type of pointed homotopies
  \begin{equation*}
    (a\mapsto\tilde{\mu}_X(\const_A))=_\ast (a\mapsto \mu(a,a))
  \end{equation*}
  is a proposition. Therefore it suffices to construct a pointed homotopy
  \begin{equation*}
    (a\mapsto\tilde{\mu}_{\mathbf{2}}(\const_A))=_\ast (a\mapsto \mu(a,a))
  \end{equation*}
  We clearly have such a pointed homotopy, since $\tilde{\mu}$ is an extension of $\mu$. Note that in this argument we made essential use of the assumption that $A$ is central. We do not currently know whether the binary operation of a unique H-space can be extended to the genuine unordered pairs.
\end{rmk}

By the previous proposition, the type \(A \to_* A\) is a set whenever \(A\) is central.
Presently we observe that it is in fact a ring.

\begin{cor} \label{cor:central-ring}
  For any central type \(A\), the set \(A \to_* A \) is a ring under pointwise multiplication and function composition.
\end{cor}

\begin{proof}
  It follows from \(A\) being a commutative and associative H-space that the set \(A \to_* A\) is an abelian group.
  The only nontrivial thing we need to show is that function composition is linear.
  Let \(f, g, \phi : A \to_* A\), and consider \(a : A\).
  By \cref{prop:unique-pmap}, \(\phi\) is an H-space map.
  Consequently,
  \[ \big( \phi \circ (f \cdot g) \big) (a) \jeq \phi( f(a) \cdot g(a) ) =  \phi(f(a)) \cdot \phi(g(a))
    \jeq \big( (\phi \circ f) \cdot (\phi \circ g) \big) (a). \qedhere \]
\end{proof}

The following remark gives some insight into the nature of the ring \(A \pto A\).

\begin{rmk}
  If \(BG\) is an \(\infty\)-group and \(X\) is a pointed type, recall that a bundle over \(X\) is \emph{\(G\)-principal} if it is classified by a map \(X \pto BG\) (see e.g.~\cite[Def.~2.23]{Sco20} for a formal definition which easily generalizes to arbitrary \(\infty\)-groups).
  In particular, it is not hard to see that the Hopf fibration of \(G\) (as the loop space of \(BG\)) is a \(G\)-principal bundle, i.e., classified by a map \(\Sigma G \pto BG \).

  In \cref{prop:loops-bautpl} we will see that any central type \(A\) has a delooping \(\BAutpl(A)\).
  This implies that we have an equivalence
  \[ (A \pto A) \ \simeq \ (\Sigma A \to_* \BAutpl(A)). \]
  Thus we see that \( A \to_* A\) is \emph{the ring of principal \(A\)-bundles over \(\Sigma A\)}.
  The equivalence above maps the identity \(\id : A \pto A\) to the Hopf fibration of \(A\) (as a principal \(A\)-bundle), meaning the Hopf fibration is the multiplicative unit from this perspective.
\end{rmk}

In the remainder of this section we collect various results which are needed later on.
The first result is that ``all'' of the evaluation fibrations of a central type \(A\) are equivalences:

\begin{prop} \label{prop:central-ev-f-equiv}
  Let $A$ be a central type and let \(f : A \to_* A\) be a pointed map.
  The evaluation fibration \(\ev_f : \pcomp{(A \to A)}{f} \to_* A\) is an equivalence, with inverse given by \(a \mapsto a \cdot f (-) \).
\end{prop}

\begin{proof}
  The type \(A \to A\) is a left-invertible H-space via pointwise multiplication,
  by \cref{prop:pmap-hspace}.
  So there is an equivalence $\pcomp{(A \to A)}{\id} \to \pcomp{(A \to A)}{f}$
  sending $g$ to $f \cdot g$.
  Since $f$ is pointed, we have
  \[ \ev_f(f \cdot g) \jeq (f \cdot g)(\pt) \jeq f(\pt) \cdot g(\pt) = \pt \cdot g(\pt) = g(\pt) = \ev_{\id}(g). \]
  In other words, $\ev_f \circ (f \cdot -) = \ev_{\id}$, which shows that $\ev_f$ is
  an equivalence.
  Since $f$ is pointed, the stated map is a section of \(\ev_f\), hence is an inverse.
\end{proof}

\begin{cor} \label{cor:ev-f-equation}
  Let $A$ be a central type, let \(f : A \to_* A\), and let \(g : \pcomp{(A \to A)}{f}\).
  Then for all \(a : A\), we have \( g(a) = g(\pt) \cdot f(a). \) \qed
\end{cor}

Any central type has an inversion map, which plays a key role in the next section.

\begin{defn} \label{defn:inversion-map}
  Suppose that \(A\) is central.
  The \textbf{inversion map} \( \dualid : A \to A\) sends $a$ to \( \dual{a} :\jeq a \linv \pt \).
\end{defn}

The defining property of $\dual{a}$ is that $a \cdot \dual{a} = \pt$.
Since $A$ is abelian, we also have $\dual{a} \cdot a = \pt$,
so it would have been equivalent to define the inversion to be \(a \mapsto \pt \rinv a\).
Because $\pt \cdot \pt = \pt$, it follows that $\dual{\pt} = \pt$,
and from commutativity of a central H-space it follows that $a^{**} = a$ for all $a$.
Thus the inversion map \(\dualid\) is a pointed self-equivalence of \(A\) and an involution.

A curious property is that on the component of \(\dualid\), inversion of equivalences is homotopic to the identity.
This comes up in the next section.

\begin{prop} \label{prop:inversion-negation-trivial}
The map \( \phi \mapsto \phi\inv : \pcomp{(A \simeq A)}{\dualid} \to \pcomp{(A \simeq A)}{\dualid} \) is homotopic to the identity.
\end{prop}

\begin{proof}
  Let \(\phi : \pcomp{(A \simeq A)}{\dualid} \).
  We need to show that \(\phi = \phi^{-1}\), or equivalently that \(\phi(\phi(\pt)) = \pt \), since \(\evid \) is an equivalence.
  (Note that \(\phi \circ \phi : \pcomp{(A \simeq A)}{\id}\).)
  Using \cref{cor:ev-f-equation}, we have that
  \[ \phi(\phi(\pt)) = \phi(\pt) \cdot \dual{\phi(\pt)} = \pt. \qedhere \]
\end{proof}

\section{Bands and torsors}\label{sec:bands-and-torsors}

We begin in \cref{ss:bands} by defining and studying
types \emph{banded} by a central type \(A\), also called \emph{\(A\)-bands}.
We show that the type $\BAutpl(A)$ of banded types is a delooping of \(A\),
that $A$ has a unique delooping,
and that every pointed self-map $A \pto A$ has a unique delooping.

In \cref{ss:tensoring-bands},
we show that $\BAutpl(A)$ is itself an H-space under a tensoring operation,
from which it follows that it is again a central type.
Thus we may iteratively consider banded types to obtain an infinite loop space structure on \(A\), which is unique.
As a special case, taking \(A\) to be \(\EM{G}{n}\) for some abelian group \(G\) produces a novel description of the infinite loop space structure on \(\EM{G}{n}\), as described in \cref{ss:EM-spaces}.

In \cref{ss:bands-and-torsors}, we define the type of \(A\)-torsors,
which we show is equivalent to the type of \(A\)-bands when \(A\) is central,
thus providing an alternate description of the delooping of $A$.
The type of $A$-torsors has been independently studied by David Wärn~\cite{Warn23},
who has shown that it is a delooping of \(A\) under the weaker assumption
that \(A\) has a unique H-space structure.

\subsection{Types banded by a central type}
\label{ss:bands}

We now study types \emph{banded} by a central type \(A\).
On this type we will construct an H-space structure, which will be seen to be central.

\begin{defn}\label{defn:bautpl}
  For a type \( A \), let \( \BAutpl(A) :\jeq \sum_{X : \Type} \, \Trunc{A = X}_0 \).
  The elements of \( \BAutpl(A) \) are types which are \textbf{banded} by $A$ or \textbf{\boldmath{\(A\)}-bands}, for short.
  We denote \(A\)-bands by \( X_p \), where \( p : \Trunc{A = X}_0 \) is the {\bf band}.
  The type \(\BAutpl(A)\) is pointed by \( A_{\trunc{\refl}_0} \).
\end{defn}

Given a band \(p : \Trunc{A = X}_0 \), we will write \(\tilde{p} : \Trunc{A \simeq X}_0\) for the associated equivalence.

\begin{rmk} \label{rmk:bautpl-locally-small}
  It's not hard to see that \( \BAutpl(A) \) is a connected, locally small type---hence essentially small, by the join construction~\cite{Rijke}.
\end{rmk}

The characterization of paths in \(\Sigma\)-types tells us what paths between banded types are.

\begin{lem}\label{lem:paths-in-bautpl}
  Consider two \(A\)-bands \(X_p\) and \(Y_q\).
  A path \(X_p = Y_q\) of \(A\)-bands corresponds to a path \( e : X = Y\) between the underlying types making the following triangle of truncated paths commute:
  \[ \begin{tikzcd}
      & \ar[dl, "p" swap] A \ar[dr, "q"] \\
      X \ar[rr, "\trunc{e}_0"] && Y \period
    \end{tikzcd} \]
  In other words, there is an equivalence
  $(X_p = Y_q) \simeq \pcomp{(X = Y)}{\bar{p} \ct q}$.  \qed
\end{lem}

For the remainder of this section, let \(A\) be a central type.
We begin by showing that the type of \(A\)-bands is a delooping of \(A\).

\begin{prop} \label{prop:loops-bautpl}
  We have that \(\Omega \BAutpl(A) \simeq_* A \).
\end{prop}

\begin{proof}
  We have \(\Omega \BAutpl(A) \simeq_* \pcomp{(A = A)}{\refl} \simeq_* \pcomp{(A \simeq A)}{\id} \simeq_* A \),
  where the first equivalence makes use of \cref{lem:paths-in-bautpl}
  and is easily seen to be pointed,
  the second equivalence is by univalence,
  and the third equivalence is by centrality.
\end{proof}

\begin{cor}\label{cor:deloopable}
  The unique H-space structure on \(A\) is deloopable. \qed
\end{cor}

Note that this gives an independent proof that it is associative (cf.\ \cref{prop:unique-associative}).

\pagebreak[2]

\begin{thm}\label{thm:uniquely-deloopable}
  The type $A$ has a unique delooping.
\end{thm}

\begin{proof}
  We must show that the type $\sum_{B} \, (\Omega B \simeq_* A)$ of deloopings of \(A\) is contractible, where \(B\) ranges over the universe of pointed, connected types.
  We will use $\BAutpl(A)$, with the equivalence $\psi$ from \cref{prop:loops-bautpl},
  as the center of contraction.
  (More precisely, we use a small type $BA$ which is equivalent to $\BAutpl(A)$,
  along with the naturally associated equivalence $\Omega(BA) \simeq_* A$,
  as the center.  See \cref{rmk:bautpl-locally-small}.  We will suppress this
  from the rest of the proof.)

  Let $B$ be a delooping of \(A\), i.e., a pointed, connected type with a pointed equivalence $\phi : \Omega B \simeq_* A$.
  Given $x : B$, consider $\pt =_B x$.
  Since $A$ is connected, $B$ is simply connected.
  Therefore, to give a banding on $\pt =_B x$, it suffices to do so when $x$ is $\pt$,
  in which case we use $\phi$.
  This defines a map $f : B \to \BAutpl(A)$.
  It is easy to show that the equivalence \(\phi : (\pt =_B \pt) \simeq A\) is an equivalence of bands, making \(f\) into a pointed map.

  We claim that the following triangle commutes:
  \[ \begin{tikzcd}
      \Omega B \ar[dr, "\phi"', "\sim"] \ar[rr, "\Omega f"] && \Omega \BAutpl(A) \ar[dl, "\psi", "\sim"'] \\
      & A \period
  \end{tikzcd} \]
  Let $q : \pt =_B \pt$.
  Then $(\Omega f)(q)$ is the path associated to the equivalence
  \[
    A \simeq (\pt =_B \pt) \simeq (\pt =_B \pt) \simeq A .
  \]
  The first equivalence is $\phi^{-1}$ and the last is $\phi$, as these give the
  pointedness of $f$.
  The middle equivalence is the map sending $p$ to $p \ct q$.
  The map $\psi$ comes from the evaluation fibration, so to compute
  $\psi((\Omega f)(q))$ we compute what happens to the base point of $A$.
  It gets sent to $\refl$, then $q$, and then $\phi(q)$.
  This shows that the triangle commutes.
  Since $A$ is an H-space, it follows from \cref{lem:unpointed-to-pointed-homotopy}
  that it is filled by a pointed homotopy.

  It follows that $\Omega f$ is an equivalence.  Since $B$ and $\BAutpl(A)$ are
  connected, $f$ is an equivalence as well.
  So $f$ and the commutativity of the triangle provide a path
  from $(B, \phi)$ to $(\BAutpl(A), \psi)$ in the type of deloopings.
\end{proof}

We conclude this section by showing how to deloop maps \(A \to_* A\).

\begin{defn}\label{defn:Bf}
  Given \(f : A \to_* A\), define \(Bf : \BAutpl(A) \to_* \BAutpl(A) \) by
  \[ Bf(X_p) :\jeq \pcomp{(X \to A)}{\dual{f} \circ \tilde{p}\inv} , \]
  where \(\dual{f} :\jeq f \circ \dualid\), and we have used that \(\dual{f} \circ \tilde{p}\inv\) is well-defined as an element of the set-truncation.
  To give a banding of \( \pcomp{(X \to A)}{\dual{f} \circ \tilde{p}\inv} \) we may induct on \(p\) and use \cref{prop:central-ev-f-equiv}.
  The same argument shows that $Bf$ is a pointed map.
\end{defn}

Note that \(f(\dual{a}) = \dual{f(a)}\) for any \(a : A\), since \(f\) is an H-space map by \cref{prop:unique-pmap}, so there's no choice involved in this definition.

Let \(g : \BAutpl(A) \pto \BAutpl(A)\).
Given a loop \(q : \pt = \pt\), the loop \((\Omega g)(q)\) is the composite
\[ \pt = g(\pt) = g(\pt) = \pt , \]
which uses pointedness of \(g\) and \(\ap_g(q)\).
We will identify \((\pt = \pt)\) with \(A\) and then write
\[ \Omega' g : A \simeq_* (\pt = \pt) \xrightarrow{\Omega g}_* (\pt = \pt) \simeq_* A. \]

\begin{prop} \label{prop:Omega-B-id}
 We have that \(\Omega' B f = f\) for any \(f : A \to_* A\).
\end{prop}

\begin{proof}
  The following diagram describes how $Bf$ acts on a loop \(p : \pt =_{\BAutpl(A)} \pt \):
  \[ \begin{tikzcd}
      A_{\refl} \dar["p", equals] && \pcomp{(A \to A)}{\dual{f}} \dar[equals, "{g \mapsto g \circ \tilde{p}\inv}"] & \lar["\sim"] A \\
      A_{\refl} && \pcomp{(A \to A)}{\dual{f}} \rar["\sim"] & A
    \end{tikzcd} \]
  Since $\tilde{p}$ is in the component of the identity,
  \cref{cor:ev-f-equation} tells us that
  $\tilde{p}(a) = x \cdot a$ for all $a : A$, where $x = \tilde{p}(\pt)$.
  So $\tilde{p}\inv(a) = x \linv a$.
  Then the composite \(A \simeq A\) on the right is seen to be
  \[ a \mapsto \ev_{\dual{f}} \bigg( \big( a \cdot \dual{f}(-) \big) \circ \tilde{p}\inv \bigg)
    = \ev_{\dual{f}} \bigg( a \cdot \dual{f} \big(x \linv (-) \big) \bigg)
    = a \cdot f(x^{**})
    =  a \cdot f(x). \]
  The domain $A_{\refl} = A_{\refl}$ is identified with $A$ by sending a path $p$
  to $\tilde{p}(\pt)$, which in this case is the $x$ above.
  The codomain $\pcomp{(A \simeq A)}{\id}$ is identified with $A$ using $\evid$,
  which sends the displayed function to $\pt \cdot f(x)$, which equals $f(x)$.
  So we have that \(\Omega Bf = f\).
  By \cref{lem:unpointed-to-pointed-homotopy}, they are equal as pointed maps.
\end{proof}

\begin{prop} \label{prop:B-Omega-id}
 We have that \(B \Omega' g = g\) for any \(g : \BAutpl(A) \to_* \BAutpl(A)\).
\end{prop}

\begin{proof}
  Given an \(A\)-band \(X_p\), we need to show that
  \( g(X_p) = \pcomp{(X \to A)}{\dual{(\Omega' g)} \circ \tilde{p}\inv}. \)
  First we construct a map of the underlying types from left to right.
  For \(y : g(X_p)\), define the map
  \[ G_y : X \xrightarrow{\sim} (\pt = X_p) \simeq (X_p = \pt) \xrightarrow{\ap_g} (g(X_p) = g(\pt)) \simeq (\pt = \pt) \to A , \]
  where the second map is path inversion, and the fourth map uses the trivialization of \(g(X_p)\) associated to \(y\) and pointedness of \(g\).
  The identification \(\pt = g(\pt)\) corresponds to a unique point \(y_0 : g(\pt)\).
  To check that \(G_y\) lies in the right component, we may induct on \(p\) and assume \(y \jeq y_0\), since \(g(\pt)\) is connected.
  We then get the map
  \[ G_{y_0} : A \xrightarrow{\dualid} A \simeq (\pt = \pt) \xrightarrow{\Omega g} (\pt = \pt) \to A, \]
  since path inversion on \((\pt = \pt)\) corresponds to inversion on \(A\), and \(y_0\) corresponds to the pointing of \(g\).
  This map is precisely the definition of \(\dual{(\Omega' g)}\), so \(G\) lands in the desired component.

  To check that \(G\) defines an equivalence of bands we may again induct on \(p\).
  Write \(\widetilde{y_0} : \pt \simeq g(\pt)\) for the equivalence associated to the point \(y_0 : g(\pt)\), which is a lift of the (equivalence associated to the) banding of \(g(\pt)\).
  It then suffices to check that the diagram
  \[ \begin{tikzcd}
      g(\pt) \ar[dr, "\widetilde{y_0}\inv" swap] \ar[rr, "G"] && \pcomp{(A \to A)}{\dual{(\Omega' g)}} \ar[dl, "{\ev_{\dual{(\Omega' g)}}}"] \\
      & \pt
    \end{tikzcd} \]
  commutes.
  Let \(y : g(\pt)\), which we identify with a trivialization \(y' : \pt = g(\pt)\).
  Chasing through the definition of \(G\) and using that \(\ap_g(\refl) = \refl\), we see that
  \[ G_y(\pt) = \ev(y' \ct \overline{y_0}) = \widetilde{y_0}\inv (y'(\pt)) \jeq \widetilde{y_0}\inv (y) , \]
  where \(\ev : (\pt = \pt) \to A\) is the last map in the definition of \(G_y\), which transports \(\pt\) along a path.
  Thus we see that the triangle above commutes, whence \(G\) is an equivalence of bands, as required.
\end{proof}

\begin{thm}\label{thm:Omega-equiv}
We have inverse equivalences
\[
  \Omega' : (\BAutpl(A) \pto \BAutpl(A)) \simeq (A \pto A) : B .
\]
In particular, the type $\BAutpl(A) \pto \BAutpl(A)$ is a set.
\end{thm}

\begin{proof}
Combine \cref{prop:Omega-B-id,prop:B-Omega-id}.
\end{proof}

\subsection{Tensoring bands}
\label{ss:tensoring-bands}

In this section, we will construct an H-space structure on $\BAutpl(A)$,
where we continue to assume that $A$ is a central type.
This H-space structure is interesting in its own right,
and also implies that $\BAutpl(A)$ is itself central.
It that follows that $A$ is an infinite loop space.
The tensor product of banded types is analogous to the classical tensor product of torsors over an abelian group.
An account of the latter is given in \cref{ss:G-torsors}.

This elementary lemma will come up frequently.

\begin{lem}\label{lem:forall-BAut1-set}
Let $P : \BAutpl(A) \to \Type$ be a set-valued type family.
Then $\prod_{X_p} P(X_p)$ is equivalent to $P(\pt)$.
\end{lem}

\begin{proof}
Since each $P(X_p)$ is a set, $\prod_{X_p} P(X_p)$ is equivalent to
$\prod_{X:\Type} \prod_{p : A = X} P(X_{\trunc{p}_0})$.
By path induction, this is equivalent to $P(A_{\trunc{\refl}_0})$, i.e., $P(\pt)$.
\end{proof}

A consequence of the following result is that any pointed \(A\)-band is trivial.

\begin{prop}\label{prop:pointed-A-band-trivial}
Let $X_p$ be an $A$-band.
Then there is an equivalence $(\pt =_{{\BAutpl(A)}} X_p) \to X$.
\end{prop}

\begin{proof}
By \cref{lem:paths-in-bautpl}, there is an equivalence
$(\pt =_{{\BAutpl(A)}} X_p) \simeq \pcomp{(A \simeq X)}{\tilde{p}}$.
We will show that $\ev_{p} : \pcomp{(A \simeq X)}{\tilde{p}} \to X$ is an equivalence.
By \cref{lem:forall-BAut1-set}, it's enough to prove this when $X_p \jeq \pt$,
and this holds because $A$ is central.
\end{proof}

We now show that path types between \(A\)-bands are themselves banded.  This underlies the main results of this section.

\begin{prop}\label{prop:path-type-banded}
  Let \(X_p\) and \(Y_q\) be \(A\)-bands.
  The type \(X_p =_{\BAutpl(A)} Y_q\) is banded by \(A\).
\end{prop}

\begin{proof}
  We need to construct a band \( \Trunc{ A = (X_p = Y_q) }_0 \).
Since the goal is a set, we may induct on \(p\) and \(q\), thus reducing the goal to \( \Trunc{ A = (\pt =_{\BAutpl(A)} \pt ) }_0 \).
Using that \((\pt =_{\BAutpl(A)} \pt) \simeq \pcomp{(A \simeq A)}{\id}\) and that \(A\) is central, we may give the set truncation of the inverse of the evaluation fibration at \(\id_A\).
\end{proof}

The following is an immediate corollary of \cref{prop:pointed-A-band-trivial}.

\begin{cor}\label{cor:X=X-trivial}
  For any \(A\)-band \(X_p\), the \(A\)-band \((X_p = X_p)\) is trivial. \qed
\end{cor}

We next define a tensor product of banded types, using the notion of duals of bands.
Write \(\invrefl : A = A\) for the self-identification of \(A\) associated to the inversion map $\dualid$ (\cref{defn:inversion-map}) via univalence.

\begin{defn}
  Let \(X_p\) be an \(A\)-band.
  The band \(\dual{p} :\jeq \trunc{\dual{\refl}} \ct p \) is the \textbf{dual of \boldmath{\(p\)}}, and the \(A\)-band \(\dual{X_p} :\jeq X_{\dual{p}}\) is the \textbf{dual of \boldmath{\(X_p\)}}.
\end{defn}

Since $\dualid$ is an involution, it follows that taking duals defines an involution on
\(\BAutpl(A)\), meaning that \(X^{**}_p = X_p\).

\begin{lem} \label{lem:negate-pt}
  We have \( \pt = \dual{\pt} \) in \( \BAutpl(A) \).
\end{lem}

\begin{proof}
  The underlying type of $\dual{\pt}$ is $A$, which has a base point, so this follows
  from \cref{prop:pointed-A-band-trivial}.
\end{proof}

We now show how to tensor types banded by \(A\).

\begin{defn} \label{defn:bautpl-tensor}
  For \( X_p,\, Y_q  : \BAutpl(A) \), define
  \( X_p \otimes Y_q :\jeq (\dual{X_p} = Y_q) \),
  with the banding from \cref{prop:path-type-banded}.
\end{defn}

It follows from \cref{lem:paths-in-bautpl} that the type \(X_p \otimes Y_q \) is equivalent to \( \pcomp{(X = Y)}{\overline{\dual{p}} \ct q}\).
Since taking duals is an involution, we also have equivalences
$X_p \otimes Y_q \jeq (\dual{X_p} = Y_q) \simeq (X_p = \dual{Y_q}) \simeq \pcomp{(X = Y)}{\overline{p} \ct \dual{q}}$.
Moreover, from \cref{cor:X=X-trivial}, we see that \(\dual{X_p} \otimes X_p = \pt\).

Tensoring defines a binary operation on \(\BAutpl(A)\), and we now show that this operation is symmetric.

\begin{prop} \label{prop:bands-tensor-symmetric}
  For any \(X_p, Y_q : \BAutpl(A) \), there is a path
  \( \sigma_{(X_p, Y_q)} : X_p \otimes Y_q =_{\BAutpl(A)} Y_q \otimes X_p \) such that \(\sigma_{\pt,\pt} = 1 \).
\end{prop}

\begin{proof}
  By univalence and the characterization of paths between bands, we begin by giving an equivalence between the underlying types.
  The equivalence will be path-inversion, as a map
  \[ \pcomp{(X = Y)}{\overline{p} \ct \dual{q}} \longrightarrow \pcomp{(Y = X)}{\overline{q} \ct \dual{p}}. \]
  To see that this is valid it suffices to show that the inversion of \(\overline{p} \ct \dual{q} \) is \(\overline{q} \ct \dual{p} \).
  We have:
  \[ \overline{\overline{p} \ct \dual{q}}
    \jeq \overline{\overline{p} \ct \dual{\refl} \ct q}
    = \overline{\dual{\refl} \ct q} \ct p
    = \overline{q} \ct \overline{\dual{\refl}} \ct p
    = \overline{q} \ct \dual{\refl} \ct p
    \jeq \overline{q} \ct \dual{p} , \]
  where we have used associativity of path composition, and that \(\overline{\invrefl} = \invrefl\) by \cref{prop:inversion-negation-trivial}.

  To prove the transport condition, we may path induct on both \(p\) and \(q\) which then yields the following triangle:
  \[ \begin{tikzcd}
      \pcomp{(A = A)}{\invrefl} \ar[dr, "\ev_{\invrefl}" swap] \ar[rr, "p \mapsto \overline{p}"] && \pcomp{(A = A)}{\invrefl} \ar[dl, "\ev_{\invrefl}"] \\
      & A \period
  \end{tikzcd} \]
  Here we are writing \( \ev_{\invrefl} \) for the composite
  \( \pcomp{(A = A)}{\invrefl} \simeq \pcomp{(A \simeq A)}{\dualid} \xrightarrow{\ev_{\dualid}} A. \)
  The horizontal map is given by path-inversion, which is homotopic to the identity by \cref{prop:inversion-negation-trivial}, hence the triangle commutes.

  Paths between paths between banded types correspond to homotopies between the underlying equivalences.
  Thus \(\sigma_{\pt, \pt} = 1\) since path-inversion on \(\pcomp{(A = A)}{\invrefl}\) is homotopic to the identity.
\end{proof}

We now use \cref{lem:symmetry} to make \(\BAutpl(A)\) into an H-space.

\begin{thm}\label{thm:BAutpl-H-space}
  The binary operation \( \otimes \) makes \( \BAutpl(A) \) into an abelian H-space.
\end{thm}

\begin{proof}
  We start by showing the left unit law.
  Since \(\dual{\pt} = \pt\), we instead consider the goal
  \( (\pt = X_p) = X_p \).
  An equivalence between the underlying types is given by \cref{prop:pointed-A-band-trivial}, which after inducting on \(p\) clearly respects the bands.
  Using \cref{prop:bands-tensor-symmetric} and \cref{lem:symmetry}, we obtain the desired H-space structure.
\end{proof}

\begin{cor}\label{cor:central-infinite-loop-space}
For a central type $A$, the type $\BAutpl(A)$ is again central.
Therefore, $A$ is an infinite loop space, in a unique way.
Moreover, every pointed map $A \pto A$ is infinitely deloopable, in a unique way.
\end{cor}

\begin{proof}
That $\BAutpl(A)$ is central follows from condition (2) of \cref{prop:central-tfae},
using \cref{thm:Omega-equiv,thm:BAutpl-H-space} as inputs.
That $A$ is a infinite loop space then follows from \cref{prop:loops-bautpl}:
writing \(\BAutpln{0}(A) :\jeq A\) and \(\BAutpln{n+1}(A) :\jeq \BAutpl(\BAutpln{n}(A))\),
we see that \(\BAutpln{n}(A)\) is an \(n\)-fold delooping of \(A\).
The uniqueness follows from \cref{thm:uniquely-deloopable}.
That every pointed self-map is infinitely deloopable in a unique way
follows by iterating \cref{thm:Omega-equiv}.
\end{proof}

Note that \(\BAutpl(A)\) is essentially small (\cref{rmk:bautpl-locally-small}),
so these deloopings can be taken to be in the same universe as $A$.

From \cref{thm:BAutpl-H-space} we deduce another characterization of central types:

\begin{prop} \label{prop:central-iff-pointed-bands-contr}
  A pointed, connected type \(A\) is central if and only if \(\sum_{X : \BAutpl(A)} \, X \) is contractible.
\end{prop}

\begin{proof}
  If \(A\) is central, then by the left unit law of~\cref{thm:BAutpl-H-space}, we have
  \[ \sum_{X : \BAutpl(A)} X \, \simeq \, \sum_{X : \BAutpl(A)} (\dual{\pt} =_{\BAutpl(A)} X) \, \simeq \, 1. \]

  Conversely, if \(\sum_{X : \BAutpl(A)} X\) is contractible, then so is its loop space.
  But the loop space is equivalent to \( \sum_{f : A \to_* A} \, \Trunc{f = \id} \), i.e., the fibre of \(\evid\) over the base point.
  Thus \(\evid\) is an equivalence, since \(A\) is connected.
\end{proof}

\subsection{Bands and torsors}
\label{ss:bands-and-torsors}

Let $A$ be a central type.
We define a notion of \(A\)-torsor which turns out to be equivalent to the notion of \(A\)-band from the previous section.
Under our centrality assumption, it follows that the resulting type of \(A\)-torsors is a delooping of \(A\).
An equivalent type of \(A\)-torsors has been independently studied by David Wärn~\cite{Warn23},
who has also shown that it gives a delooping of \(A\) under the weaker hypothesis
that \(A\) has a unique H-space structure.

\begin{defn}
An \define{action} of $A$ on a type $X$ is a map $\alpha : A \times X \to X$ such that
$\alpha(\pt, x) = x$ for all $x : X$.
If $X$ has an $A$-action, we say that it is an \define{$A$-torsor} if
it is merely inhabited and $\alpha(-,x)$ is an equivalence for every $x : X$.
The type of \define{$A$-torsor structures} on a type $X$ is
\[
  T_A(X) :\jeq \sum_{\alpha : A \times X \to X} (\alpha(\pt,-) = \id_X) \times
               \Trunc{X}_{-1} \times \prod_{x:X} \term{IsEquiv}\, \alpha(-,x) ,
\]
and the type of $A$-torsors is $\sum_{X:\Type} T_A(X)$.
\end{defn}

Since $A$ is connected, an $A$-action on $X$ is the same as a pointed map
$A \pto \pcomp{(X \simeq X)}{\id}$.  Normally one would require at a minimum
that this map sends multiplication in $A$ to composition.
We explain in \cref{rmk:action-coherent} why our definition suffices.

The condition that $\alpha(-,x)$ is an equivalence for all $x$ is equivalent
to requiring that for every $x_0,\, x_1 : X$, there exists a unique $a : A$
with $\alpha(a,x_0) = x_1$.
It is also equivalent to saying that $(\alpha, \pr_2) : A \times X \to X \times X$
is an equivalence.

For any type $X$, write $\ev_{\simeq} : (A \simeq X) \to X$ for the evaluation fibration
which sends an equivalence $e$ to $e(\pt)$.
For a map $f$, write $\term{Sect}(f)$ for the type of (unpointed) sections of $f$.

\begin{lem}
For any $X$, we have an equivalence
\[
  T_A(X) \simeq \Trunc{X}_{-1} \times \term{Sect}(\ev_{\simeq}).
\]
\end{lem}

\begin{proof}
This is simply a reshuffling of the data.
The map from left to right sends a torsor structure with action
$\alpha : A \times X \to X$ to the map $X \to (A \to X)$ sending $x$ to $\alpha(-,x)$.
By assumption, this lands in the type of equivalences, and the condition
$\alpha(\pt,-) = \id_X$ says that it is a section.
We leave the remaining details to the reader.
\end{proof}

\begin{lem}
Let $X$ be an $A$-torsor.  Then $X$ is connected.
\end{lem}

\begin{proof}
Since $X$ is merely inhabited and our goal is a proposition, we may assume
that we have $x_0 : X$.
Then we have an equivalence $\alpha(-,x_0) : A \to X$.
$A$ is connected by \cref{prop:hspace-pointed-section-comp}, so it follows that $X$ is.
\end{proof}

\begin{prop}\label{prop:band-from-torsor}
Let $X$ be an $A$-torsor.  Then $X$ is banded by $A$.
\end{prop}

\begin{proof}
Associated to the torsor structure on $X$ is a section $X \to (A \simeq X)$ of $\ev_{\simeq}$.
Since $X$ is $0$-connected, it lands in a component of $A \simeq X$.
By univalence, this determines a banding of $X$.
\end{proof}

\begin{thm}
Let $X$ be a type.  There is an equivalence $T_A(X) \simeq \Trunc{A = X}_0$.
Therefore, there is an equivalence between the type of $A$-torsors and $\BAutpl(A)$.
\end{thm}

\begin{proof}
\cref{prop:band-from-torsor} gives a map $f$.  We check that the fibres are contractible.
Let $p : \Trunc{A = X}_0$ be a banding of $X$.
An $A$-torsor structure $t$ on $X$ with $f(t) = p$ consists of a section $s$
of $\ev_{\simeq}$ that lands in the component $\pcomp{(A \simeq X)}{\tilde{p}}$,
where $\tilde{p}$ denotes the equivalence associated to $p$.
But by \cref{prop:pointed-A-band-trivial}, the evaluation fibration
$\pcomp{(A \simeq X)}{\tilde{p}} \to X$ is an equivalence, so it has a unique section.
\end{proof}

\begin{rmk}
It follows that $T_A(X)$ is a set.  One can also show this using
\cref{cor:X=X-trivial,prop:central-tfae}.
\end{rmk}

\begin{rmk}\label{rmk:action-coherent}
Let $X$ be an $A$-torsor, or equivalently, an $A$-band.
By \cref{cor:X=X-trivial}, we have an equivalence $e : A \simeq \pcomp{(X \simeq X)}{\id}$.
Since $A$ has a unique H-space structure, this equivalence is an equivalence of H-spaces,
where the codomain has the H-space structure coming from composition.
Since $A$ is connected, the $A$-action on $X$ gives a map $\alpha' : A \pto \pcomp{(X \simeq X)}{\id}$.
(In fact, $\alpha' = e$, but we won't use this fact.)
Using the equivalence $e$,
it follows from \cref{thm:Omega-equiv} that any map with the same type as $\alpha'$
is deloopable in a unique way.
That is, it has the structure of a group homomorphism in the sense of higher groups
(see~\cite{BvDR}).
This explains why our naive definition of an $A$-action is correct
in this situation.
\end{rmk}

\section{Examples and non-examples}\label{sec:examples}

We show that the Eilenberg--Mac~Lane spaces \(\EM{G}{n}\) are central whenever \(G\) is abelian and \(n > 0\),
and we use our results to give a self-contained, independent construction of Eilenberg--Mac~Lane spaces.
The base case \(\EM{G}{1}\) is discussed in \cref{ss:G-torsors} and the other
cases in \cref{ss:EM-spaces}.
In \cref{ss:products-EM-spaces}, we produce examples of products of Eilenberg--Mac~Lane spaces which are central
and examples which are not central.
At present, we do not know whether there exist central types which are not products of Eilenberg--Mac~Lane spaces.
In \cref{ss:stable-n-groups}, we show that any truncated, central type with just
two non-zero homotopy groups, both of which are finitely presented, is a product of Eilenberg--Mac~Lane spaces.

\subsection{The H-space of \texorpdfstring{\(G\)}{G}-torsors}
\label{ss:G-torsors}

Given a group \(G\), we construct the type \(TG\) of \(G\)-torsors and show that it is a \(\EM{G}{1}\).
Specifically, a pointed type $X$ \define{is a $\EM{G}{1}$}
if it is connected and comes equipped with a pointed equivalence $\Omega X \simeq_* G$ which
sends composition of loops to multiplication in $G$.
(We always point $\Omega X$ at $\refl$.)
Another account of this fact may be found in~\cite[Section~5.5]{SymmetryBook}.

When \(G\) is abelian, we can tensor \(G\)-torsors to obtain an H-space structure on \(TG\) which is analogous to the tensor product of bands of \cref{thm:BAutpl-H-space}.
These constructions are all classical and we therefore omit some details.

\begin{defn}
  Let \(G\) be a group.
  A \textbf{\boldmath{\(G\)}-set} is a set \(X\) with a group homomorphism \(\alpha : G \to \Aut(X)\).
  If the set \(X\) is merely inhabited and the map
  \( \alpha(-,x) : G \to X \)
  is an equivalence for every \(x : X\), then \((X,\alpha)\) is a \textbf{\boldmath{\(G\)}-torsor}.
  We write \(TG\) for the type of \(G\)-torsors.
  Given two \(G\)-sets \(X\) and \(Y\), we write \(X \to_G Y\) for the set of \(G\)-equivariant maps from \(X\) to \(Y\), defined in the usual way.
\end{defn}

We may write \(g \cdot x\) instead of \(\alpha(g,x)\) when no confusion can arise.
The following is straightforward to check:

\begin{lem}\label{lem:equiv-path-G-torsors}
  Let \(X\) and \(Y\) be \(G\)-torsors.
  There is a natural equivalence
  \( (X =_{TG} Y) \simeq (X \to_G Y). \)
  In particular, a \(G\)-equivariant map between \(G\)-torsors is automatically an equivalence.
  \qed
\end{lem}

Any group \(G\) acts on itself by left translation, making \(G\) into a \(G\)-torsor which constitutes the base point \(\pt\) of both \(TG\) and the type of \(G\)-sets.
Since a \(G\)-equivariant map \(\pt \to_G X\) is determined by where it sends $1 : G$,
the map \( (\pt \to_G X) \to X \) that evaluates at $1$ is an equivalence.
It is clear that the type \(TG\) is a \(1\)-type, which implies that its loop space is a group.

\begin{prop}\label{prop:loops-TG}
  We have a group isomorphism \(\Omega TG \simeq G\).
\end{prop}

  We only sketch a proof since this is a classical result.

\begin{proof}
  Since paths between \(G\)-torsors correspond to \(G\)-equivariant maps, we have equivalences of sets
  \[ (\pt =_{TG} \pt) \, \simeq \, (\pt \to_G \pt) \, \simeq \, G , \]
  where the second equivalence is given by evaluation at $1$.
  The first equivalence sends path composition to composition of maps, which reverses the order---i.e., it's an anti-isomorphism.
  The second equivalence evaluates a map at \(1 : G\).
  Thus, for \(\phi, \psi : \pt \to_G \pt\) we have
  \[ \phi( \psi(1)) = \phi( \psi(1) \cdot 1 ) = \psi(1) \cdot \phi(1) , \]
  where \(\cdot\) denotes the multiplication in \(G\).
  In other words, evaluation at \(1\) is an anti-isomorphism, meaning the composite \((\pt =_{TG} \pt) \simeq G\) is an isomorphism of groups.
\end{proof}

The following proposition says that the \(G\)-torsors are precisely those \(G\)-sets which lie in the component of the base point.

\begin{prop}
  A \(G\)-set \((X, \alpha)\) is a \(G\)-torsor if and only if there merely exists a \(G\)-equivariant equivalence from \(\pt\) to \(X\).
\end{prop}

\begin{proof}
  Suppose \(X\) is a \(G\)-torsor.
  To produce a mere \(G\)-equivariant equivalence \(\pt \simeq_G X\) we may assume we have some \(x : X\), since \(X\) is merely inhabited.
  Then \((-) \cdot x : G \to X\) yields an equivalence which is clearly \(G\)-equivariant, as required.

  Conversely, assume that there merely exists a \(G\)-equivariant equivalence from \(\pt\) to \(X\).
  Since being a \(G\)-torsor is a proposition, we may assume we have an actual \(G\)-equivariant equivalence.
  But then we are done since \(\pt\) is a \(G\)-torsor.
\end{proof}

It follows that \(TG\) is connected.
Thus by \cref{prop:loops-TG} we deduce:

\begin{cor}
  The type \(TG\) is a \(\EM{G}{1}\).  \qed
\end{cor}

For the remainder of this section, let \(G\) be an abelian group.

\begin{prop}
  For any two \(G\)-torsors \(S\) and \(T\), the path type \(S =_{TG} T\) is again a \(G\)-torsor.
\end{prop}

\begin{proof}
  First we make \(S =_{TG} T\) into a \(G\)-set.
  This path type is equivalent to the type \(S \to_G T\).
  Using that \(G\) is abelian, it's easy to check that the map
  \[ (g, \phi) \longmapsto \big( s \mapsto g \cdot \phi(s) \big) \ : \ G \times (S \to_G T) \longrightarrow (S \to_G T) \]
  is well-defined and makes \(S \to_G T\) into a \(G\)-set.

  To check that the above yields a \(G\)-torsor, we may assume that \(S \jeq \pt \jeq T\), by the previous lemma.
  One can check that \cref{prop:loops-TG} gives an equivalence of \(G\)-sets, where \(\pt \to_G \pt\) is equipped with the \(G\)-action just described.
  Thus \(\pt \to_G \pt\) is a \(G\)-torsor, as required.
\end{proof}

In order to describe the tensor product of \(G\)-torsors, we first need to define duals.

\begin{defn}
  Let \((X, \alpha)\) be a \(G\)-torsor.
  The \textbf{dual \boldmath{\(\dual{X}\)}} of \(X\) is the \(G\)-torsor \(X\) with action
  \[ \dual{\alpha}(g,x) :\jeq \alpha(g^{-1}, x). \]
\end{defn}

The tensor product of \(G\)-torsors is now defined as
\( X \otimes Y :\jeq (\dual{X} =_{TG} Y). \)

\begin{prop}
 The tensor product of \(G\)-torsors makes \(TG\) into an H-space.
\end{prop}

\begin{proof}
  We verify the hypotheses of \cref{lem:symmetry}.
  Thus our first goal is to construct a symmetry
  \[ \sigma_{X,Y} : (\dual{X} =_{TG} Y) \ =_{TG} \ (\dual{Y} =_{TG} X). \]
  After identifying paths of \(G\)-torsors with \(G\)-equivariant equivalences, we may consider the map which inverts such an equivalence.
  A short calculation shows that if $\phi : \dual{X} \to_G Y$ is $G$-equivariant,
  then $\phi^{-1} : \dual{Y} \to_G X$ is again $G$-equivariant.
  We need to check that the map sending $\phi$ to $\phi^{-1}$ is itself \(G\)-equivariant,
  so let \(\phi : \dual{X} \to_G Y\) and let \(g : G\).
  Since the inverse of \(g \cdot (-)\) is \(g^{-1} \cdot (-)\), we have:
  \[ (g \cdot \phi)\inv = \phi\inv( g\inv \cdot (-)) = g \cdot \phi\inv (-) , \]
  using that \(\phi\inv : \dual{Y} \to_G X\) is \(G\)-equivariant.
  Thus inversion is \(G\)-equivariant, yielding the required symmetry \(\sigma\).

  Now we argue that \(\sigma_{\pt, \pt} = \refl\), or, equivalently, that maps \(\dual{\pt} \to_G \pt\) are their own inverses.
  Such a map is uniquely determined by where it sends \(1 : G\), so it suffices to show that \(\phi(\phi(1)) = 1\) for every \(\phi : \dual{\pt} \to_G \pt\).
  Fortunately, we have
  \[ \phi(\phi(1)) = \phi( \phi(1) \cdot 1) = \phi(1)\inv \cdot \phi(1) = 1. \]

  Lastly, it is straightforward to check that the map
  \( (\dual{\pt} \to_G X) \to X\)
  which evaluates at \(1 : G\) is \(G\)-equivariant, for any \(G\)-torsor \(X\).
  This yields the left unit law for the tensor product \(\otimes\).
  As such we have fulfilled the hypotheses of \cref{lem:symmetry}, giving us the desired H-space structure.
\end{proof}

Using \cref{prop:central-tfae}, one can check that \(TG\) is a central H-space.
(See \cref{prop:bautpl-EM}.)

\subsection{Eilenberg--Mac~Lane spaces}
\label{ss:EM-spaces}

We now use our results to give a new construction of Eilenberg--Mac~Lane spaces.
For an abelian group $G$, recall that a pointed type \(X\) \define{is a \(\EM{G}{1}\)} if it is connected and there is a pointed equivalence \(\Omega X \simeq_* G\) which sends composition of paths to multiplication in \(G\).
For \(n > 1\), a pointed type $X$ \define{is a $\EM{G}{n+1}$}
if it is connected and $\Omega X$ is a $\EM{G}{n}$.
It follows that such an $X$ is an $n$-connected $(n+1)$-type with
$\Omega^{n+1} X \simeq_* G$ as groups.

In the previous section we saw that the type \(TG\) of \(G\)-torsors is a \(\EM{G}{1}\) and is central whenever \(G\) is abelian.
The following proposition may be seen as a higher analog of this fact.

\begin{prop} \label{prop:bautpl-EM}
  Let \( G \) be an abelian group and let \( n > 0 \).
  If a type \( A \) is a \( \EM{G}{n} \) and an H-space,
  then $A$ is central and \( \BAutpl(A) \) is a \( \EM{G}{n+1} \) and an H-space.
\end{prop}

The fact that $\BAutpl(A)$ is a $\EM{G}{n+1}$ also follows from~\cite{shulman14},
using the fact that $\BAutpl(A)$ is the 1-connected cover of $\BAut(A)$.

\begin{proof}
  Suppose that \( A \) is a \( \EM{G}{n} \) and an H-space.
  Then $A \pto \Omega A$ is contractible, since it is equivalent
  to $\Trunc{A}_{n-1} \pto \Omega A$, and $\Trunc{A}_{n-1}$ is contractible.
  So \cref{prop:central-tfae} implies that $A$ is central.
  By \cref{prop:loops-bautpl}, $\Omega \BAutpl(A) \simeq A$,
  so $\BAutpl(A)$ is a $\EM{G}{n+1}$.
  By \cref{thm:BAutpl-H-space}, $\BAutpl(A)$ is also an H-space.
\end{proof}

We can use the previous proposition to \emph{define} \(\EM{G}{n}\) for all \(n > 0\) by induction.
For the base case \(n \jeq 1\) we let \(K(G,1) :\jeq TG\), the type of \(G\)-torsors from the previous section.
When \(G\) is abelian, we saw that \(TG\) is an H-space, which lets us  apply the previous proposition.
By induction, we obtain a \(\EM{G}{n}\) for all \(n\).
Note that this construction produces a \( \EM{G}{n} \) which lives \(n-1 \) universes above the given \( \EM{G}{1} \), but that it is essentially small by the join construction~\cite{Rijke}.

\subsection{Products of Eilenberg--Mac~Lane spaces}\label{ss:products-EM-spaces}

Here is our first example of a central type that is not an Eilenberg--Mac~Lane space.

\begin{ex}
Let $K = \EM{\Zb/2}{1} = \Rb P^{\infty}$ and $L = \EM{\Zb}{2} = \Cb P^{\infty}$,
and consider $A = K \times L$.
This is a connected H-space, and
  \begin{align*}
        \big( K \times L \pto \Omega (K \times L) \big)
&\simeq \big( K \pto \Omega (K \times L) \big) & \text{since $K = \Trunc{K \times L}_1$} \\
&\simeq \big( K \pto \Omega L \big)            & \text{since $K$ is connected} \\
&\simeq \big( \Zb/2 \to_{\ab} \Zb )          & \text{by~\cite[Theorem~5.1]{BvDR}} \\
&\simeq 1 .
  \end{align*}
So it follows from \cref{prop:central-tfae}(4) that $A$ is central.
\end{ex}

On the other hand, not every product of Eilenberg--Mac~Lane spaces is central.

\begin{ex}\label{ex:GEM-not-central}
Let $K = \EM{\Zb/2}{1} = \Rb P^{\infty}$ and $L' = \EM{\Zb/2}{2}$.
A calculation like the above shows that
$K \times L' \pto \Omega (K \times L')$ is not contractible,
so $K \times L'$ is not central.

Similarly, $\EM{\Zb}{1} \times \EM{\Zb}{2})$ (i.e., $\Sb^1 \times \Cb P^{\infty}$) is not central.
Classically, this also follows from~\cref{prop:hspace-pointed-section-comp}
combined with~\cite[Proposition~Ia]{Curjel68}, which shows that $\EM{\Zb}{1} \times \EM{\Zb}{2})$
has infinitely many distinct H-space structures.
\end{ex}

Clearly both of these examples can be generalized to other groups
and shifted to higher dimensions.
Moreover, by \cref{cor:retract-of-central}, the product of any non-central type
with any pointed type is again not central.

By \cref{prop:hspace-pointed-section-comp}, centrality of a type implies that it has a unique H-space structure.
The converse fails, as we now demonstrate.
We are grateful to David Wärn for bringing our attention to this example.

\begin{ex} \label{ex:unique-H-space-not-central}
  The type \(A :\jeq \EM{\Zb}{2} \times \EM{\Zb}{3}\) is not central, by a computation similar to the one in the previous example.
  However, we note that it admits a unique H-space structure.
  Since \(A\) is a loop space it admits an H-space structure.
  Then, by the first claim in \cref{prop:stable-types}, with $k=2$,
  we see that type of H-space structures on $A$ is contractible.
\end{ex}

\subsection{Truncated types with two non-zero homotopy groups}\label{ss:stable-n-groups}

All the examples we have of central types so far are generalized Eilenberg--Mac~Lane
spaces (GEMs), i.e., products of Eilenberg--Mac~Lane spaces.
We do not know whether all central types are GEMs.
In this section we rule out a class of potential counterexamples.
Specifically, we show that any truncated central type with only two non-zero homotopy
groups, both of which are finitely presented,
is a product of Eilenberg--Mac~Lane spaces.

We first show that one can reduce to the stable range.
Let $X$ be a \( (k+1) \)-truncated central type, which is
in particular \(0\)-connected.
Since \(X\) is an infinite loop space, we may consider an \((n-1)\)-fold delooping \(B^{n-1}X\),
for any \(n>k+1\).
This is a central, \( (n-1) \)-connected, \( (n+k) \)-truncated, pointed type,
and thus represents a stable \( k\)-type, i.e., a stable \( (k+1) \)-group~\cite{BvDR}.
If \(B^{n-1}X\) is a GEM, then so is \(X\), so for our goal of ruling out non-GEM,
truncated central types,
it suffices to consider stable \((k+1)\)-groups for \(k\ge1\).

Here is the main result of the section.

\begin{thm}\label{thn:central-implies-GEM}
  Let $X$ be a truncated central type and let $n, k \geq 1$.
  Suppose that $\pi_n(X)$ and $\pi_{n+k}(X)$ are non-trivial groups
  and that all of the other homotopy groups vanish.
  Assume that $\pi_n(X)$ is finitely presented and if $k > 1$ that
  $\pi_{n+k}(X)$ is as well.
  Then $X$ is merely equivalent to $K(\pi_n(X), n) \times K(\pi_{n+k}(X), n+k)$.
\end{thm}

\begin{proof}
  Write $A :\jeq \pi_n(X)$ and $B :\jeq \pi_{n+k}(X)$.
  By the argument above, we can assume that $n > k+1$.
  Since $X$ is truncated and has no other non-trivial homotopy groups,
  the fibre of the truncation map $X \to \Trunc{X}_n \simeq \EM{A}{n}$
  is a $\EM{B}{n+k}$.
  Since we are in the stable range, we can deloop the next map in the fibre sequence,
  so we see that $X$ is the homotopy fibre of a pointed map
  \( c : \EM{A}{n} \pto \EM{B}{n+k+1} \).
  We will show that $c$ is merely homotopic to the constant map,
  which implies that $X$ splits as claimed.

  Since \(X\) is central, \(X \pto \Omega X\) is contractible,
  so \(X \pto \Omega^iX\) is connected for all \(i\ge1\).
  In particular, taking $i=k$, we get that
  \(\Hom(A,B) \simeq \Trunc{\EM An \pto \EM Bn}_0 \simeq 0\).
  Since $A$ is finitely presented and $\Zb$ is a PID (in the constructive sense),
  $A$ is merely equivalent to a finite direct sum of cyclic groups.
  (See~\cite[Theorem~V.2.3]{MRR} or~\cite[Proposition~7.3]{LQ}.)
  Our goal is a proposition, so we can assume that $A$ is explicitly
  given as such a direct sum.
  Since \(\Hom(A,B)\) is trivial and $B$ is non-trivial,
  we must have that $A$ is finite, with torsion coprime to the torsion of $B$.
  Let $r$ be the cardinality of $A$.
  Since \(X\) is deloopable, so is \(c\), and in particular, the square
    \begin{equation}\label{eq:cr-square}
      \begin{tikzcd}
        \EM An \ar[r,"c"]\ar[d,"0"'] & \EM B{n+k+1} \ar[d,"r"] \\
        \EM An \ar[r,"c"] & \EM B{n+k+1}
      \end{tikzcd}
    \end{equation}
  commutes, where we write $r$ for the map induced by multiplication by $r$ on $B$.

  Now we split into cases.
  First assume that $k > 1$, which means we also know that $B$ is finitely presented.
  Since $X \pto \Omega^{k-1} X$ is connected, we deduce that
  \(H^{n+1}(\EM{A}{n}; B) \simeq \Trunc{\EM An \pto \EM B{n+1}}_{0} \simeq 0\).
  By~\cite[Theorem~2.2.1]{CF}, the cohomology group is isomorphic to $\Ext_\Zb^1(A,B)$,
  so the latter must also vanish.
  It follows that $B$ is finite as well,
  as $\Ext^1_{\Zb}(\Zb/s, \Zb) \simeq \Zb/s$
  (see~\cite[Corollary~21]{F})
  and $\Ext^1_{\Zb}$ respects direct sums.
  Since the torsion of $A$ is coprime to the torsion of $B$,
  multiplication by $r$ on $B$ is an isomorphism.
  Therefore, the right-hand map in~\eqref{eq:cr-square} is an equivalence.
  It follows that $c$ is trivial.

  Now we consider the case when $k = 1$.
  In the square \eqref{eq:cr-square}, we no longer
  know that the map on the right is an equivalence.
  However, it does follow that $c$ factors through the fiber of $r$,
  which we analyze next.

  We claim that multiplication by $r$ on $B$ is injective.
  It suffices to show that the kernel is trivial.
  So let $b : B$ be such that $r b = 0$.
  Let $p$ be a prime factor of $r$.
  Then $p ((r/p) b) = 0$, so there is a homomorphism $A \to B$
  which hits $(r/p) b$.  So $(r/p) b = 0$.  Continuing with the remaining prime
  factors of $r$, one eventually gets that $b = 0$.

  Thus we have a short exact sequence
  \begin{equation}\label{eq:SES}
    0 \lra B \llra{r} B \llra{q} B/r \lra 0.
  \end{equation}
  This gives rise to a fibre sequence
  \[
     K(B, n+1) \llra{q} K(B/r, n+1) \llra{f} K(B, n+2) \llra{r} K(B, n+2).
  \]
  The map $c$ is a composite
  \[
    K(A, n) \llra{c'} K(B/r, n+1) \llra{f} K(B, n+2).
  \]
  The short exact sequence~\eqref{eq:SES} also gives rise
  to a six-term exact sequence ending in
  \[
    \cdots \lra \Ext^1_{\Zb}(A, B) \lra \Ext^1_{\Zb}(A, B) \lra \Ext^1_{\Zb}(A, B/r) \lra 0.
  \]
  Constructively, for general $A$, the sequence would continue with $\Ext^2$,
  but since $A$ is finitely presented, $\Ext^2_{\Zb}(A, C)$ vanishes for all $C$.
  (See~\cite{CF} and~\cite{F} for these results.)
  So the map $\Ext^1_{\Zb}(A, B) \to \Ext^1_{\Zb}(A, B/r)$ is surjective.
  Using that maps of degree one are the same as $\Ext^1_{\Zb}$,
  we can identify that map with the map
  \[
    \Trunc{\EM{A}{n} \pto \EM{B}{n+1}}_0  \lra  \Trunc{\EM{A}{n} \pto \EM{B/r}{n+1}}_0
  \]
  induced by $q$.  This means that $c'$ merely factors through $q$, and
  therefore that the composite $c = f \circ c'$ is merely zero.
\end{proof}

\setcounter{biburlnumpenalty}{500}

\printbibliography

\end{document}